\documentclass{article}
\usepackage{amsmath,amssymb,amsthm}
\usepackage{bbm}
\usepackage{graphics,graphicx}

\newtheorem{theorem}{Theorem}
\newtheorem{lemma}{Lemma}
\newtheorem{corollary}{Corollary}

\newtheorem{proposition}{Proposition}

\theoremstyle{definition}
\newtheorem{definition}{Definition}

\theoremstyle{remark}
\newtheorem{remark}{Remark}
\newtheorem{example}{Example}

\newcommand{\V}{\mathcal V}
\newcommand{\Z}{\mathbb Z}

\newcommand{\R}{\mathbb R}

\title{Crossing tribes of tangles in a thickened surface}
\author{Igor Nikonov}
\date{}

\begin{document}

\maketitle

\begin{abstract}
Crossings of knot diagrams can be divided into classes (tribes) compatible with Reidemeister moves. Tribes can be considered as localization of the notion of weak chord index introduced by M. Xu. In the article we describe tribes of crossings for tangles in a fixed surface and show that crossings are differentiated by their component, order and homotopy types. As a consequence, we conclude that there are no nontrivial indices on diagrams of classical knots.
\end{abstract}


\section*{Introduction}

%
%
%

In 2004 V.~Turaev~\cite{T} assigned an index $n(e)$ to each crossing $e$ of a flat knot diagram. Using the index, he defined a polynomial invariant (index polynomial), see also~\cite{T2}. This polynomial was numerously rediscovered in various forms afterwards~\cite{ILL,Cheng,K2,FK,Kim}. A.~Henrich~\cite{H} presented indices valued in flat link diagrams. Other types of indices can be found in~\cite{Dye,Jeong,Cheng3,KPV,CGX}. The notion of index was axiomatized by Z. Cheng~\cite{Cheng2}, and finally, M. Xu~\cite{Xu} found the general conditions for an index to define an invariant polynomial.

In~\cite{Nbm, Nif} we localized the notion of index and defined tribes as classes of crossings compatible in some sense with the Reidemeister moves (see Definition~\ref{def:crossing_tribe}). It follows from the definition that crossings which belong to one tribe have the same value for any (weak chord) index. The aim of this paper is to describe tribe of crossings and the universal index for knots, links and tangles in a fixed surface.

The paper is organized as follows. Section~\ref{sect:definitions} contains the main definitions of tangles, tribes and indices. Using technics of~\cite{Nwp}, in Section~\ref{sect:tribes_tangles} we prove the main theorem of the paper which shows that a tribe of crossings is determined by the component, order and homotopy types. This result allows to give the negative answer to the question~\cite{CGX,CFGMX} whether nontrivial indices for classical knots exist (we postpone this conclusion to Corollary~\ref{cor:classical_trivial_indices} in Section~\ref{subsect:universal_index}). Section~\ref{sect:tribes_flat_crossings} is devoted to description of tribes for flat tangles in the surface. There we also describe phratries of crossings which in some sense substitute for the sign function of crossings in the situation when we don't have it. In Section~\ref{sect:universal_index_tribe} we find the universal index for tangles and flat tangles in the surface. It appears that in order to pass from tribes to the index we need to factorize the homotopy type by the action of the inner monodromy group. We finish the paper with Section~\ref{sect:virtual_flat_knots} where we describe the universal index for diagrams of minimal genus of virtual and flat links.

\section{Definitions}\label{sect:definitions}

Let $F$ be an oriented compact connected surface with the boundary $\partial F$ (perhaps empty).

\begin{definition}\label{def:tangle}
A \emph{(oriented) tangle} in the thickened surface $F\times (0,1)$ is an embedding $T\colon M\to F\times (0,1)$ of an oriented compact $1$-dimensional manifold $M$ such that $T(\partial M)\subset \partial F\times (0,1)$ and $T$ is transversal to $\partial F$. The image of the map $T$ will be also called a tangle.

A  connected component of a tangle whose boundary is empty is called a \emph{closed component}, otherwise, the components is called \emph{long}.

If a tangle consists of one closed component then it is a \emph{knot}; a tangle consisting of one long component is a \emph{long knot}; a tangle with several component which are all closed, is a \emph{link}.

We consider tangles up to isotopies of $F\times (0,1)$ with the boundary $\partial F\times (0,1)$ fixed.

A \emph{tangle diagram} $D$ is the image of a tangle $T$ in general position under the natural projection $F\times (0,1)\to F$. Combinatorially, a tangle diagram $D$ is an embedded graph in $F$ with vertices of valency $4$ (called \emph{crossings}) and vertices of valency $1$ (forming the \emph{boundary} $\partial D$ of the diagram $D$) such that $D\cap\partial F=\partial D$ and each vertex of valency $4$ carries the structure of under-overcrossing (Fig.~\ref{pic:tangle}).

\begin{figure}[h]
\centering
  \includegraphics[width=0.4\textwidth]{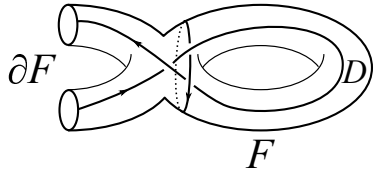}
  \caption{A tangle diagram}\label{pic:tangle}
\end{figure}
\end{definition}

Any tangle diagrams which correspond to isotopical tangles can be connected by a sequence of isotopies of the surface $F$ identical on $\partial F$, and \emph{Reidemeister moves} (Fig.~\ref{pic:reidmove}).

\begin{figure}[h]
\centering
  \includegraphics[width=0.8\textwidth]{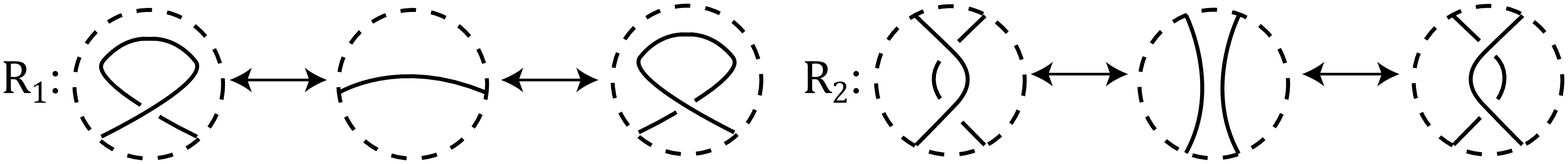}\\ \vbox{\phantom{1em}}
  \includegraphics[width=0.3\textwidth]{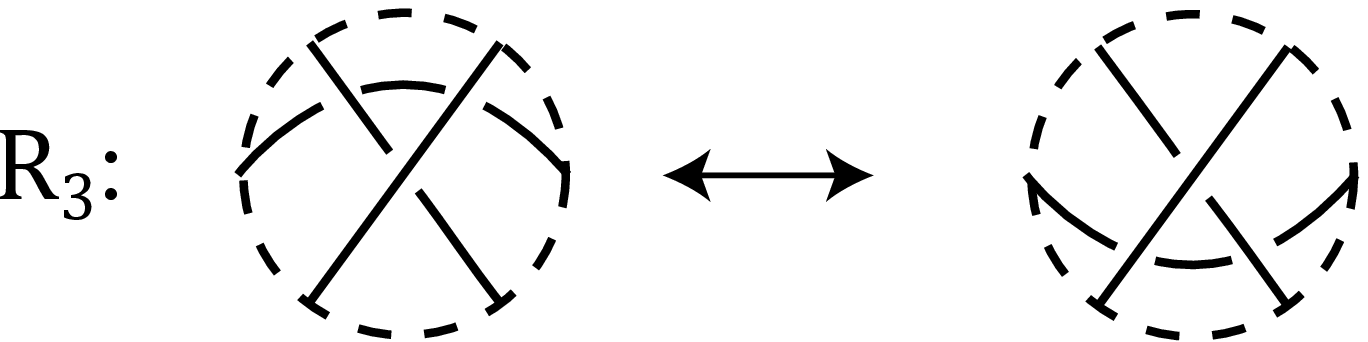}
  \caption{Reidemeister moves}\label{pic:reidmove}
\end{figure}

Let $D$ be a tangle diagram. Denote the set of its crossings by $\V(D)$. We assume that all tangle diagrams are oriented, hence, any crossing $v\in\V(D)$ has a \emph{sign} (Fig.~\ref{pic:crossing_sign}).

\begin{figure}[h]
\centering\includegraphics[width=0.25\textwidth]{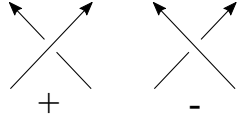}
\caption{The sign of a crossing}\label{pic:crossing_sign}
\end{figure}

Let $D=D_1\cup\cdots\cup D_n$ be a diagram of a tangle $T$ and $v$ be a self-crossing of some component $D_i$. Then the oriented smoothing at $v$ splits the component $D_i$ into two the \emph{left half} $D^l_v$ and the \emph{right half} $D^r_v$  (Fig.~\ref{pic:knot_halves}). We define also the \emph{signed halves} $D^\pm_v$ by the formula $D^{sgn(v)}_v=D^l_v$ and $D^{-sgn(v)}_v=D^r_v$.

\begin{figure}[h]
\centering\includegraphics[width=0.4\textwidth]{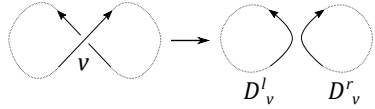}
\caption{The left and the right halves of the diagram}\label{pic:knot_halves}
\end{figure}

Note that $D^+_v$ is the half of $D_i$ from the undercrossing of $v$ to the overcrossing of $v$, and $D^-_v$ is the half from the overcrossing of $v$ to the undercrossing of $v$.

\subsection{Tribes and indices}

Let $T$ be a tangle and $\mathfrak T$ be the set of its diagram. The set $\mathfrak T$ can be considered as the objects of a diagram category whose morphisms are compositions of isotopies and Reidemeister moves.

Given tangle diagrams $D$ and $D'$ and a morphism $f\colon D\to D'$ between them, there is a correspondence $f_*\colon\V(D)\to\V(D')$ between the crossings of the diagrams. The map $f_*$ is a partial bijection between the crossing sets because some crossing can disappear or appear when $f$ is a first or second Reidemeister move.

\begin{remark}
Below we will often denote the crossing $f_*(v)$ in the diagram $D'$ which corresponds to a crossing $v$ in $D$, by the same letter $v$ when the morphism $f$ is clear from the context.
\end{remark}

Recall the notion of a tribal system~\cite{Nbm,Nif}.

\begin{definition}\label{def:crossing_tribe}
Let $T$ be a tangle and $\mathfrak T$ be the category of its diagrams. A \emph{tribal system} on diagrams of the tangle $T$ is a family of partitions $\mathfrak C(D)$ of the sets $\V(D)$, $D\in\mathfrak T$, into subsets $C\in\mathfrak C(D)$ called \emph{tribes} such that
\begin{itemize}
  \item[(T0)] Let $f\colon D\to D'$ be a Reidemeister move, $v_1,v_2\in\V(D)$ and $v'_1=f_*(v_1)$ and $v'_2=f_*(v_2)$ be the correspondent crossings in $D'$. Then if $v_1$ and $v_2$ belong to one tribe in $\mathfrak C(D)$ then $v'_1$ and $v'_2$ also belong to one tribe in $\mathfrak C(D')$;
  \item[(T2)] If $f\colon D\to D'$ is a decreasing second Reidemeister move and $v_1$ and $v_2$ are the disappearing crossings then $v_1$ and $v_2$ belong to one tribe in $\mathfrak C(D)$.
\end{itemize}
\end{definition}

Below in the paper, we will always consider the \emph{finest tribal system} $\mathfrak C^u$ on the diagram category $\mathfrak T$, i.e. the equivalence relations on the crossing sets $\V(D)$, $D\in\mathfrak T$, generated by the properties (T0) and (T2).

For crossings $v_1,v_2\in\V(D)$, $D\in\mathfrak T$, we denote $v_1\sim v_2$ if the crossings $v_1,v_2$ belong to one tribe.

The sign function splits any tribe $C\in\mathfrak C(D)$, $D\in\mathfrak T$, into two subsets $C^+$ and $C^-$ we call \emph{phratries}. By definition
\[
C^\pm=\{v\in C\mid sgn(v)=\pm 1\}.
\]
The phratries $C^+$ and $C^-$ are called \emph{dual}.

The notion of tribes is a localization of the notion a (weak chord) index defined by M. Xu~\cite{Xu} (see also~\cite{Cheng2} for the original axiomatic definition of chord index, and~\cite{CGX,CFGMX} for examples of chord indices).

\begin{definition}\label{def:index}
Let $I$ be an arbitrary set. An \emph{index on the diagrams of a tangle $T$ in the surface $F$ with coefficients in the set $I$} is a map $\iota$ which assigns some value $\iota(v)$ to any crossing $v$ in any diagram $D$ of the tangle $T$ and possesses the following properties:
\begin{itemize}
\item[(I0)] for any Reidemeister move $f\colon D\to D'$ and any crossings $v\in\V(D)$ and $v'\in\V(D')$ such that $v'=f_*(v)$, one has $\iota(v)=\iota(v')$;
\item[(I2)] $\iota(v_1)=\iota(v_2)$ for any crossings $v_1,\,v_2\in\V(D)$ to which a decreasing second Reidemeister move can be applied.
\end{itemize}
\end{definition}

\begin{example}
Let $K$ be a knot in the surface $F$. For any diagram $D$ of $K$ and any crossing $v\in\V(D)$, define its index to be the intersection number of two cycles: $Ind(v)=D\cdot D^+_v\in\Z$. Then the map $Ind$ obeys the conditions (I0) and (I2).
\end{example}

The following statement follows from the definitions.
\begin{proposition}\label{prop:index_and_tribes}
Let $\iota$ be an index on the diagrams of a tangle $T$ in the surface $F$ with coefficients in a set $I$. Let $v$ and $w$ be two vertices of a diagram $D$ of $T$. If $v$ and $w$ belong to one tribe then $\iota(v)=\iota(w)$.
\end{proposition}

\begin{definition}\label{def:universal_index}
An index $\iota^u$ on the diagrams of a tangle $T$ in the surface $F$ with coefficients in a set $I^u$ is called the \emph{universal index} on diagrams of the tangle $T$ if for any index $\iota$ on $T$ with coefficients in a set $I$ there is a unique map $\psi\colon I^u\to I$ such that $\iota=\psi\circ\iota^u$.
\end{definition}

The main result of the paper is a description of the tribes (and phratries) of the finest tribal system $\mathfrak C$ on diagrams of any tangle $T$ in the surface $F$, and the universal index on $T$.

\section{Crossing tribes of a tangle in the surface $F$}\label{sect:tribes_tangles}

Let $T=K_1\cup K_2\cup\cdots\cup K_n$ be tangle with $n$ components in the surface $F$ and $\mathfrak T$ be its diagram category.

\begin{definition}
Let $D=D_1\cup\cdots\cup D_n$ be a diagram of the tangle $T$ and $v\in\V(D)$ be a crossing of $D$. The \emph{component type} $\tau(v)$ of the crossing $v$ is the pair $(i,j)$ where $D_i$ is the overcrossing component at $v$ and $D_j$ is the undercrossing component at $v$ (Fig.~\ref{pic:crossing_type}).

\begin{figure}[h]
\centering\includegraphics[width=0.1\textwidth]{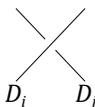}
\caption{A crossing of type $(i,j)$}\label{pic:crossing_type}
\end{figure}

Note that there are $n^2$ types of crossings.
\end{definition}

The following statement is evident.
\begin{proposition}\label{prop:tribe_type}
Let $v_1, v_2\in\V(D)$, $D\in\mathfrak T$. If $v_1\sim v_2$ then $\tau(v_1)=\tau(v_2)$.
\end{proposition}

\subsection{An auxiliary lemma}

Let us formulate an auxiliary lemma.

\begin{lemma}\label{lem:main_lemma}
Let $D$ be a tangle diagram in $F$. Let $a,b\in\V(D)$ be the ends of two curves $\gamma_1,\gamma_2\subset D$ such that (Fig.~\ref{pic:homotopical_paths})
\begin{enumerate}
  \item $\gamma_1$ and $\gamma_2$ intersect transversely (i.e. they do not have common edges);
  \item $\gamma_2$ overcrosses $\gamma_1$ at the crossings $a$ and $b$;
  \item $\gamma_1$ is homotopic to $\gamma_2$ in $F$ as a curve with fixed ends.
\end{enumerate}

\begin{figure}[h]
\centering\includegraphics[width=0.25\textwidth]{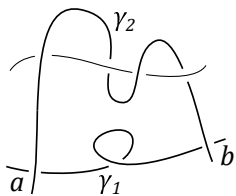}
\caption{Paths $\gamma_1$ and $\gamma_2$}\label{pic:homotopical_paths}
\end{figure}

Then $a$ and $b$ belong to one tribe.
\end{lemma}

\begin{proof}

The proof follows the reasoning of~\cite{Nwp}.

Without loss of generality, we assume that $\gamma_1$ and $\gamma_2$ are oriented from $a$ to $b$.

Step 1.  Pull the crossing $b$ to the crossing $a$ as shown in Fig.~\ref{pic:crossing_pull}. The pulling process consists second and third Reidemeister moves and defines a morphism $f\colon D\to D'$. The pulling contracts $\gamma_1$ and transforms the path $\gamma_2$ to a path $\gamma'_2$ in $D'$ which is homotopically equivalent to $\gamma_2\gamma_1^{-1}$. Since $\gamma_1$ is homotopic to $\gamma_2$ then $\gamma'_2$ can be considered as a contractible loop.

\begin{figure}[h]
\centering\includegraphics[width=0.5\textwidth]{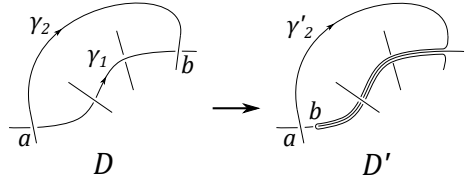}
\caption{Step 1: pulling the crossing $b$ to the crossing $a$ }\label{pic:crossing_pull}
\end{figure}

By the property (T0) $a\sim b$ in $D$ if and only if $a\sim b$ in $D'$. Thus, we reduce the situation to the case when $\gamma_1$ is a small arc and $\gamma_2$ is a contractible loop in $F$ (up to contraction $\gamma_1$ to a point).

Step 2. If $a$ and $b$ have the same sign we apply a first, a second and a third Reidemeister move to create the crossing $c$ as shown in Fig.~\ref{pic:crossing_reverse}.
\begin{figure}[h]
\centering\includegraphics[width=0.5\textwidth]{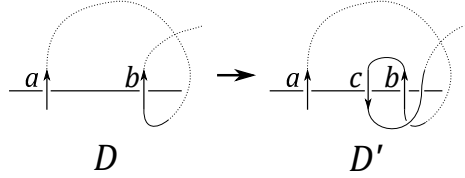}
\caption{Step $2$: changing the crossing sign if necessary}\label{pic:crossing_reverse}
\end{figure}

By the property (T2) $b\sim c$ in $D'$. Then it is suffice to prove that $a\sim c$: in this case $a\sim c\sim b$ in $D'$, hence $a\sim b$ in $D$ by the property (T0). Thus, below we will assume the $a$ and $b$ have different signs.

Step 3. There is a homotopy which contracts $\gamma_2$ to a small arc with ends $a$ and $b$ (Fig.~\ref{pic:contracting_homotopy}).

\begin{figure}[h!]
\centering\includegraphics[width=0.5\textwidth]{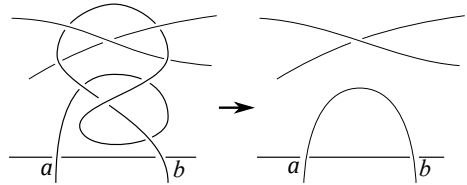}
\caption{Contracting homotopy}\label{pic:contracting_homotopy}
\end{figure}

The contraction process can be considered as a homotopy of the curve in the thickening $F\times (0,1)$. The homotopy can be represented as a sequence of isotopies in $F\times (0,1)$ and self-intersections which look like a change of under- and overcrossing in some crossing (see
Fig.~\ref{pic:cross_change}) after projection to $F$.
\begin {figure}[ht]
\centering
\includegraphics[width=0.3\textwidth]{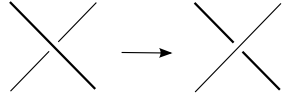}
\caption{Change of crossing}\label{pic:cross_change}
\end {figure}

Then we replace each crossing change by an isotopy which adds a small loop (see Fig.~\ref{pic:cross_change1}).

\begin {figure}[ht]
\centering
\includegraphics[width=0.4\textwidth]{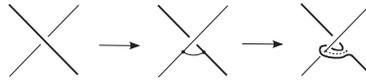}
\caption{Adding a loop instead of crossing change}\label{pic:cross_change1}
\end {figure}

As a result we get an isotopy whose projection transforms the curve $\gamma_2$ to a small arc with "sprouts" (see Fig.~\ref{pic:contraction} middle).

\begin{figure}[ht]
\centering\includegraphics[width=0.9\textwidth]{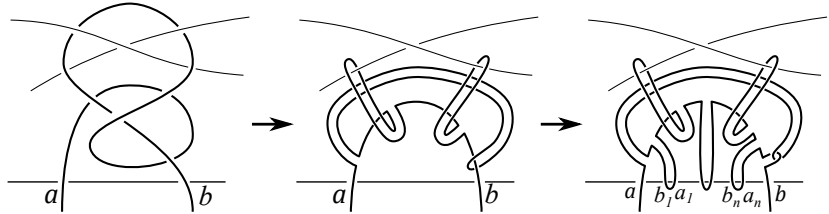}
\caption{Contraction of $\gamma_2$}\label{pic:contraction}
\end{figure}

Then we apply second Reidemeister moves to each arc separating the bases of sprouts (Fig.~\ref{pic:contraction} right) and to the arc $ab$. Let us denote the appearing crossings by $b_1, a_1,\dots, b_n, a_n$. By the property (T2) we have $b_1\sim a_1$, \dots, $b_n\sim a_n$.

Step 4. Let us prove that $a\sim b_1$, $a_1\sim b_2$,\dots, $a_n\sim b$.

Let $a$ and $b$ be the roots of a sprout which grows outside the small loop which $\gamma_2$ was contracted to (Fig.~\ref{pic:reduction1} left). Then take the segment $ab$ and pull it under the diagram along the sprout to its tip. Then in the diagram obtained (Fig.~\ref{pic:reduction1} right) we have $a\sim b$ by the property (T2). And by the property (T0), $a\sim b$ in the initial diagram.
\begin{figure}[h]
\centering\includegraphics[width=0.65\textwidth]{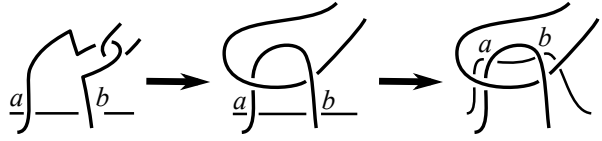}
\caption{A sprout growing outside}\label{pic:reduction1}
\end{figure}

Let $a$ and $b$ be the roots of a sprout which grows outside the small loop (Fig.~\ref{pic:reduction2} top left). Then pull the arc $ab$ up to the base of the sprout (Fig.~\ref{pic:reduction2} top left). For the appearing crossings $c, c'$, $d, d'$, we have $c\sim c'$ and $d\sim d'$ by the property (T2) (Fig.~\ref{pic:reduction2} bottom left). We pull the arc $ab$ to the beginning of the sprout and see (Fig.~\ref{pic:reduction2} bottom right) that $a\sim c'$ and $b\sim d'$ by the property (T2). Then $a\sim c$ and $b\sim d$. After pulling the arc $cd$ to the tip of the sprout we get the equivalence $c\sim d$ by the property (T2) (Fig.~\ref{pic:reduction2} top right). Hence, $a\sim b$.
\begin{figure}[h]
\centering\includegraphics[width=0.7\textwidth]{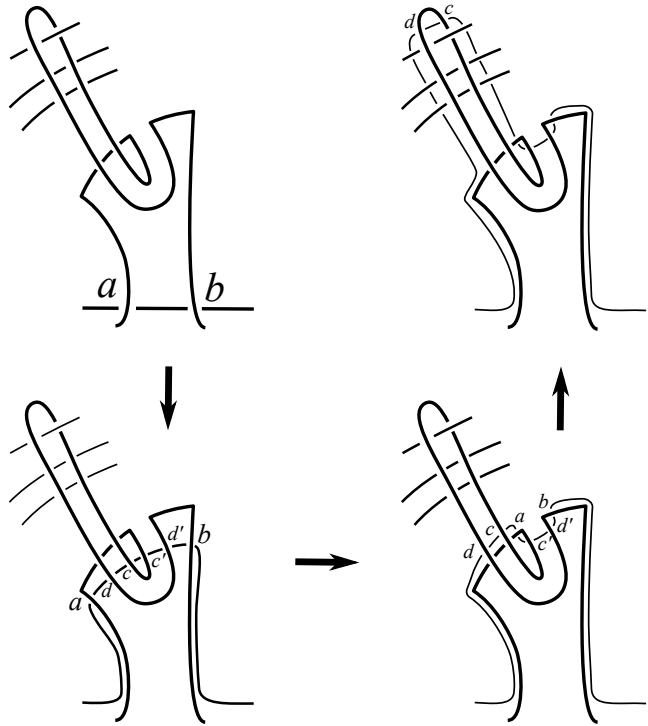}
\caption{A sprout growing inside}\label{pic:reduction2}
\end{figure}

Applying the reasoning above to every sprout, we get $a\sim b_1$, $a_1\sim b_2$,\dots $a_n\sim b$. On the other hand, $b_1\sim a_1$, \dots, $b_n\sim a_n$. Then $a\sim b$. The lemma is proved.
\end{proof}

\clearpage

\subsection{Self-crossings of a long component}

Let $D_i$ be a component of the tangle diagram $D$ such that $\partial D_i\ne\emptyset$. Let $s\in\partial D_i$ be the beginning point of the oriented component $D_i$.

\begin{definition}\label{def:order_type}
A self-crossing $v$  of the component $D_i$ is called an \emph{early undercrossing} if we pass the undercrossing of $v$ before its overcrossing while moving along the component $D_i$ from $s$. Otherwise, the crossing $v$ is an \emph{early overcrossing}.

\begin{figure}[h]
\centering\includegraphics[width=0.5\textwidth]{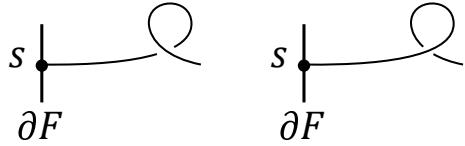}
\caption{An early undercrossing (left) and an early overcrossing (right) }\label{pic:early_underovercrossing}
\end{figure}

Define the \emph{order type} $o$ of the crossing $v$ to be $o(v)=-1$ if $v$ is an early undercrossing, and $o(v)=+1$ if $v$ is an early overcrossing.
\end{definition}

The following proposition is checked directly.

\begin{proposition}\label{prop:order_type_invariance}
Let $D_i$ be a long component of the tangle diagram $D$ and $v_1$ and $v_2$ be self-crossing of $D_i$ (i.e. $\tau(v_1)=\tau(v_2)=(i,i)$). If $v_1\sim v_2$ then $o(v_1)=o(v_2)$.
\end{proposition}

\begin{definition}
Let $v$ be a self-crossing $v$  of the component $D_i$. Let $\gamma$ be the arc of $D_i$ from the beginning point $s$ to the first appearance of $v$, and $\alpha$ be the arc from the first appearance of $v$ to the second (Fig.~\ref{pic:long_homotopy_type}). The \emph{homotopy type} of the crossing $v$ is the homotopy class
\begin{equation}\label{eq:long_homotopy_type}
h(v)=[\gamma\alpha\gamma^{-1}]\in\pi_1(F,s).
\end{equation}

\begin{figure}[h]
\centering\includegraphics[width=0.3\textwidth]{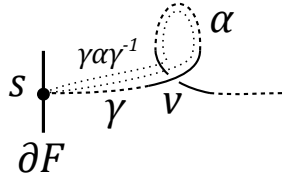}
\caption{The homotopy type of a crossing $v$}\label{pic:long_homotopy_type}
\end{figure}
\end{definition}

\begin{proposition}\label{prop:long_selfcrossing}
Let $v$ and $w$ be self-crossings of the long component $D_i$ of one order type, i.e. $\tau(v)=\tau(w)=(i,i)$ and $o(v)=o(w)$. Then $v\sim w$ if and only if $h(v)=h(w)$.
\end{proposition}

\begin{proof}
By definition the homotopy type of a crossing $v$ does not change under isotopies and Reidemeister moves which do not eliminate $v$.

Let $v$ and $w$ be crossings to which a second Reidemeister move can be applied (Fig.~\ref{pic:homotopy_type_r2}). Denote the arc of $D_i$ from the beginning point $s$ to $v$ by $\gamma$, and the arc between the appearances of $w$ by $\alpha$. We ignore the small arcs between $v$ and $w$ by contracting them to a point. Then $h(v)=h(w)=[\gamma\alpha\gamma^{-1}]$.

\begin{figure}[h]
\centering\includegraphics[width=0.3\textwidth]{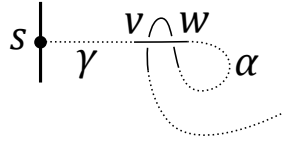}
\caption{Invariance of homotopy type $h$ under a second Reidemeister move}\label{pic:homotopy_type_r2}
\end{figure}

The invariance of the homotopy type established above means that all crossings of one tribe have the same homotopy type. Thus, $v\sim w$ implies $h(v)=h(w)$.

Now, assume that $v$ and $w$ are early overcrossings and $h(v)=h(w)$. Pull the undercrossings of $v$ and $w$ to the starting point $s$ and get a diagram $D'$ (Fig.~\ref{pic:long_homotopy_type_proof}). The pulling process does not change the homotopy types  $h(v)$ and $h(w)$. Assume that the undercrossing of $v$ preceeds the undercrossing of $w$ in the component $D_i$. Denote the arc of $D'$ between the overcrossing of $w$ and the undercrossing of $v$ by $\alpha$ and the arc between the undercrossing of $v$ and the undercrossing of $w$ by $\beta$.

\begin{figure}[h]
\centering\includegraphics[width=0.3\textwidth]{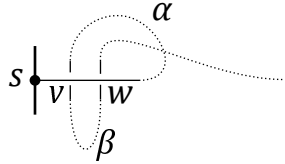}
\caption{Diagram $D'$}\label{pic:long_homotopy_type_proof}
\end{figure}

Then $h(v)=[\alpha]$ and $h(w)=[\alpha\beta]$. Since $h(v)=h(w)$, we have $[\beta]=1$, that is the arc $\beta$ is contractible. By Lemma~\ref{lem:main_lemma}, $v\sim w$ in the diagram $D'$. Then by the property (T0) $v\sim w$ in $D$.

The case of early undercrossings is proved analogously.
\end{proof}

\subsection{Self-crossings of a closed component}

Let $D_i$ be a closed component of the tangle diagram $D$, i.e. $\partial D_i=\emptyset$.

Fix an arbitrary non-crossing point $z$ in $D_i$.

Let $v$ be a self-crossing on the component $D_i$. Denote the arc of $D_i$ which connects $z$ with the overcrossing of $v$ and does not contain the undercrossing of $v$, by $\gamma^o_{v,z}$; the arc which connects $z$ with the undercrossing of $v$ and does not contain the overcrossing of $v$, by $\gamma^u_{v,z}$. Let $\alpha_{z,v}$ be the arc between the under- and overcrossings of $v$ which does not include $z$ (Fig.~\ref{pic:based_diagram}). We orient the curves $\gamma^o_{v,z}$ and $\gamma^u_{v,z}$ from $z$ to $v$, and the arc $\alpha_{z,v}$ keeps the orientation of the component $D_i$.

\begin{figure}[h]
\centering\includegraphics[width=0.3\textwidth]{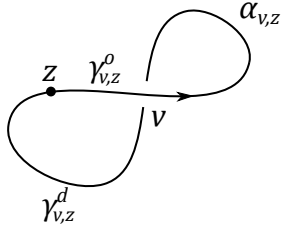}
\caption{Arcs $\gamma^o_{v,z}$, $\gamma^u_{v,z}$ and $\alpha_{z,v}$}\label{pic:based_diagram}
\end{figure}

Recall that $D^+_v$ is the half of $D_i$ from the undercrossing of $v$ to the overcrossing of $v$, and $D^-_v$ starts in the overcrossing of $v$ and ends in the undercrossing of $v$. The crossing $v$ is called a \emph{$z$-overcrossing} if $z\in D^+_v$ (i.e. one passes first the overcrossing of $v$ when moves from $z$ along $D_i$), and a \emph{$z$-undercrossing} if $z\in D^-_v$.

Denote the homotopy class of the component $D_i$ by $\kappa_z=[D_i]\in\pi_1(F,z)$. Let $\bar\pi_D(F,z)$ be the quotient of $\pi_1(F,z)$ by the adjoint action of the element $\kappa_z$. Then $\bar\pi_D(F,z)$ consists of orbits $\bar x=\{\kappa_z^n x\kappa_z^{-n} \mid n\in\Z\}$, $x\in\pi_1(F,z)$.

\begin{definition}\label{pic:closed_homotopy_type}
Let $v$ be a self-crossing of the component $D_i$. The loops $\hat D^+_{v,z}=\gamma^u_{v,z}D^+_v(\gamma^u_{v,z})^{-1}$ and $\hat D^-_{v,z}=\gamma^o_{v,z}D^-_v(\gamma^o_{v,z})^{-1}$ are called the \emph{based signed halves} of the crossing $v$ with the base point $z$ (Fig.~\ref{pic:based_halves}). Let us denote the homotopy classes of the based signed halves $D^\pm_{v,z}$ in $\pi_1(F,z)$ by $\delta^\pm_{v,z}$.

\begin{figure}[h]
\centering\includegraphics[width=0.3\textwidth]{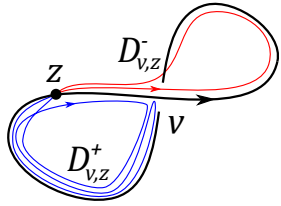}
\caption{Based signed halves}\label{pic:based_halves}
\end{figure}

The orbit
\begin{equation}\label{eq:closed_homotopy_type}
h_{D,z}(v)=\bar\delta^+_{v,z}\in\bar\pi_D(F,z)
\end{equation}
of the homotopy class of the based half $\hat D^+_{v,z}$ is called the \emph{homotopy type} of the crossings $v$.
\end{definition}

\begin{lemma}\label{lem:homotopy_type_basepoint_change}
1. Let $v$ be a self-crossing of a closed component $D_i$. If $v$ is a $z$-overcrossing then $\delta^-_{v,z}\delta^+_{v,z}=\kappa_z$, if $v$ is a $z$-undercrossing then $\delta^+_{v,z}\delta^-_{v,z}=\kappa_z$.

2. Let $v$ and $w$ be a self-crossings of a closed component $D_i$ and $z,z'$ be points in $D_i$ which are not crossings. Then $h_{D,z}(v)=h_{D,z}(w)$ if and only if $h_{D,z'}(v)=h_{D,z'}(w)$.
\end{lemma}

\begin{proof}
1. Let $v$ be a $z$-overcrossing. Then $D^+_v=(\gamma^u_{v,z})^{-1}\gamma^o_{v,z}$ and $D^-_v=\alpha_{v,z}$. Hence, $\delta^+_{v,z}=[\gamma^u_{v,z}(\gamma^u_{v,z})^{-1}\gamma^o_{v,z}(\gamma^u_{v,z})^{-1}]=[\gamma^o_{v,z}(\gamma^u_{v,z})^{-1}]$ and
$\delta^-_{v,z}=[\gamma^o_{v,z}\alpha_{v,z}(\gamma^o_{v,z})^{-1}]$. Thus,
\[
\delta^-_{v,z}\delta^+_{v,z}=[\gamma^o_{v,z}\alpha_{v,z}(\gamma^o_{v,z})^{-1}\gamma^o_{v,z}(\gamma^u_{v,z})^{-1}]=
[\gamma^o_{v,z}\alpha_{v,z}(\gamma^u_{v,z})^{-1}]=\kappa_z.
\]
The proof for a $z$-undercrossing is analogous.

2. Let $\gamma_{z,z'}$ is an arc of $D_i$ from $z$ to $z'$. Then the formula $x\mapsto(\gamma_{z,z'})^{-1}x\gamma_{z,z'}$ defines an isomorphism $\phi\colon \pi_1(F,z)\to\pi_1(F,z')$. Since $\phi(\kappa_z)=\kappa_{z'}$, the isomorphism $\phi$ induces a bijection $\bar\phi\colon\bar\pi_D(F,z)\to \bar\pi_D(F,z')$.

Since the paths $\gamma^u_{v,z'}$ and $(\gamma_{z,z'})^{-1}\gamma^u_{v,z}$ connects $z'$ with $v$, we have $(\gamma_{z,z'})^{-1}\gamma^u_{v,z}=\kappa_{z'}^p\gamma^u_{v,z'}$ for some $p\in\Z$. Then
\begin{multline*}
\bar\phi(h_{D,z}(v))=\overline{[(\gamma_{z,z'})^{-1}\gamma^u_{v,z}D^+_v(\gamma^u_{v,z})^{-1}\gamma_{z,z'}]}=
\overline{[\kappa_{z'}^{p}\gamma^u_{v,z'}D^+_v(\gamma^u_{v,z'})^{-1}\kappa_{z'}^{-p}]}=\\
\overline{[\gamma^u_{v,z'}D^+_v(\gamma^u_{v,z'})^{-1}]}=h_{D,z'}(v).
\end{multline*}

Analogously, $\bar\phi(h_{D,z}(w))=h_{D,z'}(w)$. Since $\bar\phi$ is a bijection, $h_{D,z}(v)=h_{D,z}(w)$ if and only if $h_{D,z'}(v)=h_{D,z'}(w)$.
\end{proof}

\begin{proposition}\label{prop:closed_selfcrossings}
Let $v$ and $w$ be self-crossings of the closed component $D_i$. Then $v\sim w$ if and only if $h_{D,z}(v)=h_{D,z}(w)$.
\end{proposition}

\begin{proof}
1.  Let us show first that $v\sim w$ implies $h_{D,z}(v)=h_{D,z}(w)$. Define an equivalence relation on the crossings of the component $K_i$ of the tangle $T$ for all diagrams of $T$ by the formula: $v\sim_h w$ if and only if $h_{D,z}(v)=h_{D,z}(w)$, $v,w\in\V(D_i)\subset\V(D)$, $D\in\mathfrak T$. By Lemma~\ref{lem:homotopy_type_basepoint_change}, the equivalence relation $\sim_h$ does not depend on the choice of the point $z$.

  Let us check that the relation $\sim_h$ obeys the properties (T0) and (T2).

  Let $f\colon D\to D'$ be an isotopy or a Reidemeister move. We can assume that $f$ does not move the point $z$ (by Lemma~\ref{lem:homotopy_type_basepoint_change} we can choose any appropriate point). Then $\kappa_z=\kappa'_z\in\pi_1(F,z)$. Hence, the sets of orbits $\bar\pi_D(F,z)$ and $\bar\pi_{D'}(F,z)$ are identical. For any crossing $v$ in the component $D_i$ which survives under the move $f$, we have $\delta^+_{v,z}=(\delta')^+_{f_*(v),z}\in\pi_1(F,z)$, hence, $h_{D,z}(v)=h_{D',z}(f_*(v))$ where $f_*(v)$ is the correspondent vertex in $D'_i$, and $(\delta')^+_{f_*(v),z}$ is the homotopy class of $(\widehat{D'})^+_{f_*(v),z}$. Thus, the condition $v\sim_h w$ is equivalent to the condition $f_*(v)\sim_h f_*(w)$. So, the property (T0) is fulfilled.

  Let $v,w$ be two self-crossings of $D_i$ to which a second Reidemeister move can be applied (Fig.~\ref{pic:closed_homotopy_type_T2}). Then $[\hat D^+_{v,z}]=[\hat D^+_{w,z}]\in\pi_1(F,z)$. Hence, $h_{D,z}(v)=h_{D,z}(w)$. Thus, the property (T2) holds.

\begin{figure}[h]
\centering\includegraphics[width=0.3\textwidth]{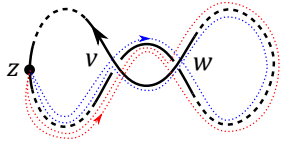}
\caption{Based signed halves $\hat D^+_{v,z}$ (blue) and  $\hat D^+_{w,z}$ (red)}\label{pic:closed_homotopy_type_T2}
\end{figure}

Since, the equivalence relation $\sim$ is the finest with the properties (T0) and (T2), the condition $v\sim w$ implies $v\sim_h w$, i.e. $h_{D,z}(v)=h_{D,z}(w)$.

2. Now, assume that $h_{D,z}(v)=h_{D,z}(w)$.

If $\delta^+_{v,z}=\delta^+_{w,z}$ for some base point $z$ then pull the overcrossings of $v$ and $w$ to the point $z$ along the arcs $\gamma^u_{v,z}$ and $\gamma^u_{w,z}$ and get a new diagram $D'$ (Fig.~\ref{pic:closed_homotopy_type_proof1}).

\begin{figure}[h]
\centering\includegraphics[width=0.35\textwidth]{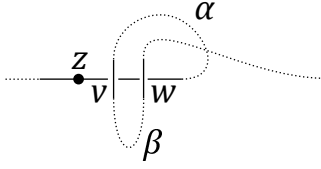}
\caption{Diagram $D'$}\label{pic:closed_homotopy_type_proof1}
\end{figure}

Then $\delta^+_{v,z}=(\delta')^+_{v,z}=[\alpha]$ and $\delta^+_{w,z}=(\delta')^+_{w,z}=[\alpha\beta]$. Since $\delta^+_{v,z}=\delta^+_{w,z}$, the arc $\beta$ is contractible. Then $v\sim w$ in $D'$ by Lemma~\ref{lem:main_lemma}, and $v\sim w$ in $D$ by the property (T0).

Assume then that $\delta^+_{w,z}=\kappa_z^p\delta^+_{v,z}\kappa_z^{-p}$. We can suppose that the undercrossing of $v$ is close to $z$ so $\hat D^+_{v,z}=\alpha_{v,z}$ up to arcs in a contractible neighbourhood of $z$ (Fig.~\ref{pic:closed_homotopy_type_proof2} left). For brevity denote the arc $\alpha_{v,z}$ by $\alpha$.

\begin{figure}[h]
\centering\includegraphics[width=0.7\textwidth]{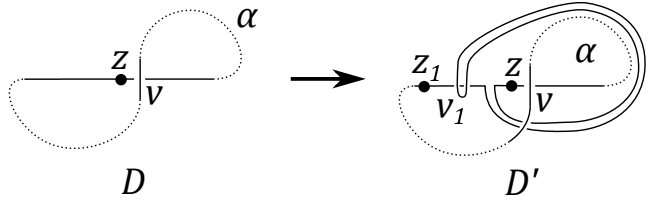}
\caption{Construction of the crossing $v_1$}\label{pic:closed_homotopy_type_proof2}
\end{figure}

Choose a point $z_1$ on $D_i$ near $z$ and pull a part of the segment $zz_1$ along the loop $\alpha$ and create a crossing $v_1$ as shown in Fig.~\ref{pic:closed_homotopy_type_proof2} right. Then $(\delta')^+_{v_1,z'}=[\alpha]$ and $(\delta')^+_{v,z'}=[\alpha\alpha^{-1}\alpha\alpha^{-1}\alpha]=[\alpha]$. Then as we proved above, $v\sim v_1$ in $D'$. On the other hand, $(\delta')^+_{v,z}=\delta^+_{v,z}$ and $(\delta')^+_{v_1,z}=[\kappa'_z\alpha(\kappa'_z)^{-1}]=\kappa_z\delta^+_{v,z}\kappa_z^{-1}$.

Repeating the process, we will get a diagram $\tilde D$ with new crossings $v_1, v_2,\dots, v_p$ such that $v_i\sim v_{i+1}$ in $\tilde D$ and $\tilde\delta^+_{v_i,z}=\kappa^i\delta^+_{v,z}\kappa^{-i}$. Hence, $v_p\sim v$ in $\tilde D$. On the other hand, $\tilde\delta^+_{v_p,z}=\kappa^i\delta^+_{v,z}\kappa^{-i}=\delta^+_{w,z}=\tilde\delta^+_{w,z}$, therefore, $v_n\sim w$ in $\tilde D$. Thus, $v\sim w$ in $\tilde D$ and $v\sim w$ in $D$ by the property (T0).

The case $p<0$ can be proved analogously.
\end{proof}


\subsection{Mixed crossings}

Let $D_i$ and $D_j$ be two components of the tangle diagram $D$.

Choose a point $z_i$ in the component $D_i$ to be equal to the beginning point $s_i$ if the component $D_i$ is not closed, and an arbitrary non crossing point if $D_i$ is closed. By the same rule, choose a point $z_j$ in $D_j$.

Let $v$ be a crossing of the components $D_i$ and $D_j$. Denote the arc in $D_i$ from $z_i$ to $v$ by $\gamma_{z_i,v}$, and the arc in $D_i$ from $z_j$ to $v$ by $\gamma_{z_j,v}$ (Fig.~\ref{pic:mixed_crossing_paths})

\begin{figure}[h]
\centering\includegraphics[width=0.3\textwidth]{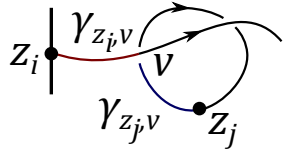}
\caption{Paths $\gamma_{z_i,v}$ and $\gamma_{z_j,v}$}\label{pic:mixed_crossing_paths}
\end{figure}

Let $\kappa_i\in\pi_1(F,z_i)$ (respectively, $\kappa_j\in\pi_1(F,z_j)$) be the homotopy class of $D_i$ (respectively, the homotopy class of $D_j$) if it is closed and $1$ otherwise.

Let $\pi_1(F,z_i,z_j)$ be the set of homotopy classes of continuous paths in $F$ from $z_i$ to $z_j$. Then the fundamental group $\pi_1(F,z_i)$ acts on this set from the left, and the group $\pi_1(F,z_j)$ acts on it from the right. Consider the set $\bar\pi_D(F,z_i,z_j)=\kappa_i\backslash\pi_1(F,z_i,z_j)/\kappa_j$ consisting of cosets
\[
\bar x=\{\kappa_i^p x\kappa_j^q \mid p,q\in\Z\},\ x\in\pi_1(F,z_i,z_j).
\]

\begin{definition}\label{def:mixed_homotopy_type}
Let $v$ be a crossing of the components $D_i$ and $D_j$. The \emph{homotopy type} of the crossing $v$ is the element
\begin{equation}\label{eq:mixed_homotopy_type}
h_{D,z_i,z_j}(v)=\overline{[\gamma_{z_i,v}\gamma_{z_j,v}^{-1}]}\in\bar\pi_D(F,z_i,z_j).
\end{equation}
\end{definition}

\begin{proposition}\label{prop:mixed_crossings}
Let $v$ and $w$ be crossings of the components $D_i$ and $D_j$ with $D_i$ overcrossing $D_j$. Then $v\sim w$ if and only if $h_{D,z_i,z_j}(v)=h_{D,z_i,z_j}(w)$.
\end{proposition}

\begin{proof}
Assume that $D_i$ is closed and $z_i$, $z_i'$ be two non crossing points in $D_i$. Denote the path from $z_i$ to $z_i'$ by $\gamma_{z_i,z_i'}$. The map $x\mapsto (\gamma_{z_i,z_i'})^{-1}x$ defines a bijection $\phi\colon\pi_1(F,z_i,z_j)\to\pi_1(F,z_i',z_j)$ and a bijection between the sets $\bar\phi\colon\bar\pi_D(F,z_i,z_j)\to\bar\pi_D(F,z_i',z_j)$.

For a crossing $v$ of the components $D_i$ and $D_j$, we have $\gamma_{z_i',v}=\kappa_i^p(\gamma_{z_i,z_i'})^{-1}\gamma_{z_i,v}$, hence $\bar\phi(h_{D,z_i,z_j}(v))=h_{D,z_i',z_j}(v)$. Thus, the condition $h_{D,z_i,z_j}(v)=h_{D,z_i,z_j}(w)$ is equivalent to the condition $h_{D,z_i',z_j}(v)=h_{D,z_i',z_j}(w)$. Analogously, we can change the point $z_j$ if the component $D_j$ is closed.

For any crossings $v$ and $w$ of component type $(i,j)$ in a diagram $D$ of the tangle $T$,  denote $v\sim_h w$ if and only if $h_{D,z_i,z_j}(v)=h_{D,z_i,z_j}(w)$. We prove that the equivalence relation $\sim_h$ obeys (T0) and (T2).

Let $f\colon D\to D'$ be an isotopy or a Reidemeister move. We can suppose that $f$ does not move points $z_i$ and $z_j$ (otherwise choose other base points). Then the sets $\bar\pi_D(F,z_i,z_j)$ and $\bar\pi_{D'}(F,z_i,z_j)$ coincide. The paths $\gamma_{z_i,v}\gamma_{z_j,v}^{-1}$ in $D$ and $\gamma'_{z_i,v}(\gamma'_{z_j,v})^{-1}$ in $D'$ are homotopic. Hence, for any crossing $v$ of component type $(i,j)$ we have $h_{D,z_i,z_j}(v)=h_{D',z_i,z_j}(f_*(v))$ if the correspondent crossings $f_*(v)$ in $D'$ exists. Thus, the condition $v\sim_h w$ in $D$ implies $f_*(v)\sim_h f_*(w)$ in $D$. The property (T0) is fulfilled.

Let $v$ and $w$ be two crossings of component type $(i,j)$ to which one can apply a second Reidemeister move (Fig.~\ref{pic:mixed_crossing_T2}).

\begin{figure}[h]
\centering\includegraphics[width=0.3\textwidth]{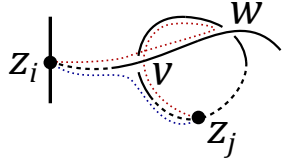}
\caption{Mixed crossings in a second Reidemeister move}\label{pic:mixed_crossing_T2}
\end{figure}

The paths $\gamma_{z_i,v}\gamma_{z_j,v}^{-1}$ and $\gamma_{z_i,w}\gamma_{z_j,w}^{-1}$ are homotopic, hence $v\sim_h w$. Thus, the property (T2) holds.

Since $\sim$ is the finest equivalence with the properties (T0) and (T2), $v\sim w$ implies $v\sim_h w$, i.e. $h_{D,z_i,z_j}(v)=h_{D,z_i,z_j}(w)$.

Now, let $h_{D,z_i,z_j}(v)=h_{D,z_i,z_j}(w)$. Assume first that $[\gamma_{z_i,v}\gamma_{z_j,v}^{-1}]=[\gamma_{z_i,w}\gamma_{z_j,w}^{-1}]$.

Pull the crossings $v$ and $w$ to the point $z_i$ along the paths $\gamma_{z_i,v}$ and $\gamma_{z_i,w}$ and get a diagram $D'$ (Fig.~\ref{pic:mixed_homotopy_type_proof1}). Let $\gamma$ be the arc in $D_j$ from $z_j$ to $w$ and $\alpha$ be the arc from $w$ to $v$.

\begin{figure}[h]
\centering\includegraphics[width=0.3\textwidth]{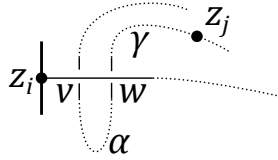}
\caption{Diagram $D'$}\label{pic:mixed_homotopy_type_proof1}
\end{figure}

Then $[\gamma_{z_i,v}\gamma_{z_j,v}^{-1}]=[\gamma\alpha]$ and $\gamma_{z_i,w}\gamma_{z_j,w}^{-1}]=[\gamma]$. By assumption, $[\gamma\alpha]=[\gamma]$, hence, $\alpha$ is contractible. Then $v\sim w$ by Lemma~\ref{lem:main_lemma}.

Let $[\kappa_i^p\gamma_{z_i,v}\gamma_{z_j,v}^{-1}]=[\gamma_{z_i,w}\gamma_{z_j,w}^{-1}]$, $p\in\Z$. Then pull the crossings $v$ and $w$ to the point $z_i$ along the paths $\kappa_i^p\gamma_{z_i,v}$ and $\gamma_{z_i,w}$, and apply the arguments of the previous case. Then we get $v\sim w$. Analogously, the equality $[\gamma_{z_i,v}\gamma_{z_j,v}^{-1}\kappa_j^q]=[\gamma_{z_i,w}\gamma_{z_j,w}^{-1}]$, $q\in\Z$ implies $v\sim w$.

In general case, $[\kappa_i^p\gamma_{z_i,v}\gamma_{z_j,v}^{-1}\kappa_j^q]=[\gamma_{z_i,w}\gamma_{z_j,w}^{-1}]$, $p,q\in\Z$. With a second Reidemeister move, create a crossing $v_1$ near $v$ such that $[\gamma_{z_i,v_1}]=[\kappa_i^p\gamma_{z_i,v}]$ and $[\gamma_{z_j,v_1}]=[\gamma_{z_j,v}]$ (we identify homotopy classes of paths from $z_i$ and $z_j$ to $v$ and $v_1$ by merging $v$ with $v_1$), see Fig.~\ref{pic:mixed_homotopy_type_proof2}. Then $v\sim v_1$ and $v_1\sim w$ by the previous cases. Thus, $v\sim w$.
\begin{figure}[h]
\centering\includegraphics[width=0.4\textwidth]{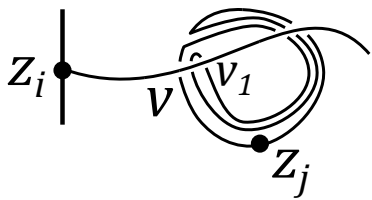}
\caption{Crossing $v_1$}\label{pic:mixed_homotopy_type_proof2}
\end{figure}
\end{proof}

\begin{remark}\label{rem:mixed_inverse_homotopy_type}
Instead of the element $h_{D,z_i,z_j}(v)$ we can take the element
\[
h_{D,z_j,z_i}(v)=\overline{[\gamma_{z_j,v}\gamma_{z_i,v}^{-1}]}\in\bar\pi_D(F,z_j,z_i)
\]
where $\bar\pi_D(F,z_j,z_i)=\kappa_j\backslash\pi_1(F,z_j,z_i)/\kappa_i$, as the homotopy type of a mixed crossing $v$. Since the classes $h_{D,z_i,z_j}(v)$ and $h_{D,z_j,z_i}(v)$ are inverse to each other, the condition $h_{D,z_i,z_j}(v)=h_{D,z_i,z_j}(w)$ is equivalent to the condition $h_{D,z_j,z_i}(v)=h_{D,z_j,z_i}(w)$.
\end{remark}

\subsection{Main theorem}

We can summarize Propositions~\ref{prop:long_selfcrossing},~\ref{prop:closed_selfcrossings} and~\ref{prop:mixed_crossings} in the following theorem.

\begin{theorem}\label{thm:main_theorem}
Let $T$ be an oriented tangle in the surface $F$ and $D$ be its diagram. Then two crossings $v$ and $w$ of the diagram $D$ belong to one tribe if and only if they have identical component type, order type (if $v$ and $w$ are self-crossings of a long component) and homotopy type.
\end{theorem}

For the phratries of the crossings of a tangle we get immediately the following description.
\begin{corollary}\label{cor:phratries_surface_tangle}
Let $T$ be an oriented tangle in the surface $F$ and $D$ be its diagram. Then two crossings $v$ and $w$ of the diagram $D$ belong to one phratry if and only if they have identical signs, component type, order type (if $v$ and $w$ are self-crossings of a long component) and homotopy type.
\end{corollary}

\begin{corollary}\label{cor:classical_trivial_tribes}
Let $K$ be a classical knot. Then all the crossings of a diagram of $K$ belong to one tribe.
\end{corollary}
\begin{proof}
For a classical knots we have $F=\R^2$ or $S^2$. In both cases $\pi_1(F,z)=1$, hence, the homotopy type $h$ is trivial. For a knot, the component type $\tau$ is trivial too. And we have no order type. Thus, all crossing belong to one tribe.
\end{proof}

We postpone the further corollaries of Theorem~\ref{thm:main_theorem}  to Section~\ref{subsect:universal_index}.

\section{Tribes and phratries of flat tangles in the surface $F$}\label{sect:tribes_flat_crossings}

\emph{Flat tangles} in the surface $F$ can be defined as immersions of an oriented compact 1-manifold into the thickened surface $F\times(0,1)$ up to homotopies fixed on the boundary $\partial F\times(0,1)$. Combinatorially, one defines flat tangles as equivalence classes of tangle diagrams modulo isotopies, Reidemeister moves (Fig.~\ref{pic:reidmove}) and crossing changes (Fig.~\ref{pic:cross_change}).

Flat tangle diagrams are obtained from tangle diagrams by forgetting under-overcrossing structure (Fig.~\ref{pic:flat_tangle}). A flat diagram can be considered as the equivalence class of a tangle diagram modulo crossing change operations applied to the crossings of the diagram. We identify the crossing sets of all the diagrams in this equivalence class.

\begin{figure}[h]
\centering
  \includegraphics[width=0.4\textwidth]{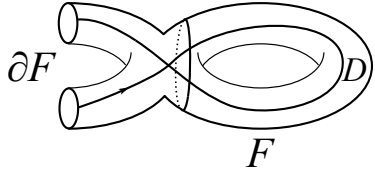}
  \caption{A flat tangle diagram}\label{pic:flat_tangle}
\end{figure}

\begin{proposition}\label{prop:types_crossing_change}
Let $D=D_1\cup\cdots\cup D_n$ be a diagram of a tangle $T$ in the surface $F$ and $v\in\V(D)$ be a crossing. After the crossing change applied to $v$ we get a tangle diagram $\bar D$ with the same set of crossing as $D$.
\begin{enumerate}
  \item Let $w\in\V(D)=\V(\bar D)$ be a crossing different from $v$. Then the component, order and homotopy types $\tau_{\bar D}(w)$, $o_{\bar D}(w)$, $h_{\bar D}(w)$ of the crossing $w$ in the diagram $\bar D$ coincide with the correspondent types $\tau_{D}(w)$, $o_{D}(w)$, $h_{D}(w)$  of $w$ in the diagram $D$:
\[
\tau_{\bar D}(w)=\tau_{D}(w),\quad o_{\bar D}(w)=o_{D}(w),\quad h_{\bar D}(w)=h_{D}(w).
\]
  \item Let $\tau_D(v)=(i,j)$, $o_{D}(v)$, $h_{D}(v)$ be the types of the crossing $v$ in $D$. Then
  \begin{itemize}
    \item $\tau_{\bar D}(v)=(j,i)$;
    \item $o_{\bar D}(v)=-o_D(v)$ if $i=j$ and the component $D_i$ is long;
    \item the homotopy type $h_{\bar D}(v)$ is equal to:
    \begin{itemize}
      \item $h_{D}(v)\in\pi(F,s_i)$ if $i=j$, $\partial D_i\ne\emptyset$ and $s_i$ is the beginning point of the component $D_i$;
      \item $h_D(v)\in\bar\pi(F,z_i,z_j)$ if $i\ne j$ (by Remark~\ref{rem:mixed_inverse_homotopy_type} we can choose the set $\bar\pi(F,z_i,z_j)$ for the homotopy index)
      \item $\kappa_i h_{D,z_i}^{-1}(v)\in\bar\pi_1(F,z_i)$ if $i=j$, $\partial D_i=\emptyset$ and $z_i\in D_i\setminus\V(D)$
    \end{itemize}
  \end{itemize}
\end{enumerate}
\end{proposition}
\begin{proof}
  The proposition follows from the definitions. For example, if $v$ is a self-crossing of a closed component $D_i$ then its homotopy type in $\bar D$ is
\[
h_{\bar D}(v)=h_{\bar D,z_i}(v)=\bar\delta'^+_{v,z}=\bar\delta^-_{v,z}=\kappa_i(\bar\delta^+_{v,z})^{-1}=\kappa_i h_{D,z_i}^{-1}(v),
\]
since $D'^+_v=D^-_v$ and $\delta^+_{v,z}\delta^-_{v,z}$ or $\delta^-_{v,z}\delta^+_{v,z}$ is equal to $\kappa_i$ by Lemma~\ref{lem:homotopy_type_basepoint_change}.
\end{proof}

Let $T=K_1\cup\cdots\cup K_n$ be a flat tangle in the surface $F$ and $\mathfrak T$ be the category of flat diagrams of $T$. Let us describe the finest tribal system on $T$ (the definition of tribes for flat diagrams repeats verbally Definition~\ref{def:crossing_tribe}).

\begin{definition}\label{def:flat_types}
Let $D=D_1\cup\cdots\cup D_n$ be a diagram of the flat tangle $T$. Let $v$ be a crossing of the diagram $D$. Then $v$ is an intersection point of components $D_i$ and $D_j$.

The \emph{flat component type} of the crossing $v$ is the unordered pair $\bar\tau(v)=\{i,j\}$.

Let $i=j$, the component $D_i$ be long, and $s_i$ be the beginning point of $D_i$. The \emph{flat homotopy type} $\bar h(v)$ of the crossing $v$ coincides with the homotopy type $h(v)\in\pi_1(F,s_i)$ defined by the formula~\eqref{eq:long_homotopy_type}.

Let $i=j$ and the component $D_i$ be closed. The \emph{flat homotopy type} $\bar h(v)$ of the crossing $v$ is the element $h_{D,z}(v)$ defined by the formula~\eqref{eq:closed_homotopy_type} and considered as an element in the quotient set
\[
\bar{\bar\pi}_D(F,z)=\bar\pi_D(F,z)/\sigma.
\]
where $\sigma$ is the involution $\sigma(x)=\kappa_i x^{-1}$, $x\in\bar\pi_D(F,z)$.

Let $i\ne j$. The \emph{flat homotopy type} $\bar h(v)=\bar h_{D,z_i,z_j}(v)$ of the crossing $v$ coincides with the homotopy type $h_{D,z_i,z_j}(v)\in\bar\pi_D(F,z_i,z_j)$ defined by the formula~\eqref{eq:mixed_homotopy_type}.
\end{definition}

Now we can formulate an analogue of Theorem~\ref{thm:main_theorem} for flat tangles.

\begin{theorem}\label{thm:tribes_flat_tangle}
Let $T$ be an oriented flat tangle in the surface $F$ and $D$ be its diagram. Then two crossings $v$ and $w$ of the diagram $D$ belong to one tribe if and only if they have identical flat component types and flat homotopy types.
\end{theorem}

\begin{proof}
  By Proposition~\ref{prop:types_crossing_change}, if $v$ and $w$ belong to one tribe of flat crossings then their flat component type and flat homotopy type coincide.

Let $\bar\tau(v)=\bar\tau(w)$ and $\bar h(v)=\bar h(w)$. Then lift the diagram $D$ to a diagram $\tilde D$ of a tangle by choosing arbitrarily the under-overcrossing structure in all the crossings of $D$ except $v$ and $w$. For the crossings $v$ and $w$ we choose the under-overcrossings by the following rule:
\begin{itemize}
  \item Let $v$ and $w$ are mixed crossings of components $D_i$ and $D_j$, $i<j$. Then choose the under-over crossings so that $\tau(v)=\tau(w)=(i,j)$ in the diagram $\tilde D$.
  \item Let $v$ and $w$ are self-crossings of a long component $D_i$. Then choose the under-over crossings so that $o(v)=o(w)=+1$ in the diagram $\tilde D$.
  \item Let $v$ and $w$ are self-crossings of a closed component $D_i$. Let $\alpha\in\bar\pi_D(F,z)$ be one of the two elements which correspond to the element $h(v)=h(w)\in\bar{\bar\pi}_D(F,z)$. Then choose the under-over crossings so that $\bar\delta^+_{v,z}=\bar\delta^+_{w,z}=\alpha$ in the diagram $\tilde D$.
\end{itemize}
Thus, we obtain a tangle diagram $\tilde D$ with crossings $v$ and $w$ which have identical component, order and homotopy types. Then by Theorem~\ref{thm:main_theorem} $v$ and $w$ belong to one tribe of crossings of tangle diagrams. This means that there exists a sequence of intermediate diagrams $\tilde D,\tilde D_1,\dots,\tilde D_n,\tilde D$ and crossings in them which establishes the equivalence $v\sim w$ by sequental application of properties (T0) and (T2). Then the flattening $D, D_1,\dots, D_n, D$ of this sequence prove the equivalence of $v$ and $w$ as crossings of flat tangle diagrams.
\end{proof}

Crossings of flat diagrams have no signs, so we can not split tribes into phratries like we did in Definition~\ref{def:crossing_tribe}. But the difference between two parts of a tribe reveals when one considers an  oriented parity of crossings~\cite{Npf}: parity values of two crossings in a second Reidemeister move are opposite although the crossings belong to one tribe. Thus, we need another definition of tribe splitting.

\begin{definition}\label{def:phratry_system}
Let $T$ be a (flat) tangle and $\mathfrak T$ be the category of its diagrams. A \emph{phratry system} on diagrams of the tangle $T$ is a family of partitions $\mathfrak P(D)$ of the sets $\V(D)$, $D\in\mathfrak T$, into subsets $P\in\mathfrak P(D)$ called \emph{phratries} such that
\begin{itemize}
  \item[$(\Phi^\ast)$] There is an involution $\ast\colon\mathfrak P(D)\to\mathfrak P(D)$. For a phratry $P\in\mathfrak P(D)$, the phratry $P^*$ is called the \emph{dual phratry} to $P$; if $P^*=P$ then the phratry $P$ is called \emph{self-dual}. The union $P\cup P^*$ is called a \emph{tribe};
  \item[$(\Phi0)$] Let $f\colon D\to D'$ be a Reidemeister move, $v_1,v_2\in\V(D)$ and $v'_1=f_*(v_1)$ and $v'_2=f_*(v_2)$ be the correspondent crossings in $D'$. Then if $v_1$ and $v_2$ belong to one phratry in $\mathfrak P(D)$ then $v'_1$ and $v'_2$ also belong to one phratry in $\mathfrak P(D')$; if $v_1$ and $v_2$ belong to dual phratries then $v'_1$ and $v'_2$ also belong to dual phratries in $\mathfrak P(D')$;
  \item[$(\Phi2)$] If $f\colon D\to D'$ is a decreasing second Reidemeister move and $v_1$ and $v_2$ are the disappearing crossings then $v_1$ and $v_2$ belong to dual phratries in $\mathfrak P(D)$.
\end{itemize}
\end{definition}

\begin{remark}
1. A tribal system $\mathfrak C$ in Definition~\ref{def:crossing_tribe} defines a phratry system with phratries $C^\pm$, $C\in\mathfrak C(D)$, $D\in \mathfrak T$. The phratries $C^+$ and $C^-$ are dual.

2. Any phratry system $\mathfrak P$ induces a tribal system with tribes $P\cup P^*$, $P\in\mathfrak P(D)$, $D\in\mathfrak T$.
\end{remark}

Below we will always consider the \emph{finest phratry system} $\mathfrak P^u$, i.e. the partitions of the crossing sets $\V(D)$, $D\in\mathfrak T$, generated by the properties $(\Phi0)$ and $(\Phi2)$.

\begin{remark}\label{rem:phratry_graph}
Let us give a graph interpretation of the finest phratry and tribal systems. Let $T$ be a (flat) tangle in the surface $F$. We construct a family of simple unoriented graphs $G(D)$ whose vertex sets $V(G(D))$ are the crossing sets $\V(D)$, $D\in\mathfrak T$. The sets of edges $E(G(D))$ of the graphs $G(D)$ are defined inductively by the following rules:
\begin{itemize}
  \item ($R_2$-duality) If $v_1$ and $v_2$ are crossings in a diagram $D$ to which a decreasing second Reidemeister move can be applied then $v_1v_2\in E(G(D))$;
  \item (Reidemeister translation) Let $f\colon D\to D'$ be a Reidemeister move, $v_1,v_2\in\V(D)$ and $v'_1=f_*(v_1)$ and $v'_2=f_*(v_2)$ be the correspondent crossings in $D'$. If $v_1v_2\in E(G(D))$ then $v'_1v'_2\in E(G(D'))$;
  \item (transitive closure) Let $v_1$, $v_2$ and $v_3$ be crossings of a diagram $D\in\mathfrak T$ such that $v_1v_2, v_2v_3\in E(G(D))$. Then $v_1v_3\in E(G(D))$.
\end{itemize}

Then for each diagram $D\in\mathfrak T$ the graph $G(D)$ defines a partition of the crossing set $\V(D)$ into subsets which are the vertex sets of the connected components of $G(D)$. This partition is the partition $\mathfrak C^u(D)$ of the finest tribal system $\mathfrak C^u$.

In order to describe the finest phratries, let us assign weights $\epsilon(v_1v_2)\in\Z_2$ to the edges $v_1v_2\in E(G(D))$, $D\in\mathfrak T$, of the graphs $G(D)$. We define weights inductively by the following rules:
\begin{itemize}
  \item ($R_2$-duality) If $v_1$ and $v_2$ are crossings in a diagram $D$ to which a decreasing second Reidemeister move can be applied then $\epsilon(v_1v_2)=1$;
  \item (Reidemeister translation) Let $f\colon D\to D'$ be a Reidemeister move, $v_1,v_2\in\V(D)$ and $v'_1=f_*(v_1)$ and $v'_2=f_*(v_2)$ be the correspondent crossings in $D'$. If $v_1v_2\in E(G(D))$ then $\epsilon(v'_1v'_2)=\epsilon(v_1v_2)$;
  \item (transitive closure) Let $v_1$, $v_2$ and $v_3$ be crossings of a diagram $D\in\mathfrak T$ such that $v_1v_2, v_2v_3\in E(G(D))$. Then $\epsilon(v_1v_3)=\epsilon(v_1v_2)+\epsilon(v_2v_3)$.
\end{itemize}

For some components of the graphs $G(D)$, $D\in\mathfrak T$, it may happen that the weights are not consistent. For example, consider the flat diagram $D$ in Fig.~\ref{pic:self-dual_tribe}. It has three crossings $u,v,w$. One can eliminate one of the crossings with a first Reidemeister move and apply a second Reidemeister move to the other two. Then $G(D)$ is the triangle $uvw$ and the weights of all the edges are equal to $1$. This contradicts the transitive closure rule of weight assignment.

\begin{figure}[h]
\centering
  \includegraphics[width=0.15\textwidth]{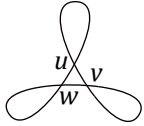}
  \caption{A diagram with a self-dual tribe}\label{pic:self-dual_tribe}
\end{figure}

If we encounter a contradiction of weights for some component $G'\subset G(D)$ then we call the component $G'$ \emph{self-dual} and its vertex set $C'=V(G')$ is called a \emph{self-dual tribe}. Then the tribe $C'$ consists of one phratry of the finest phratry system $\mathfrak P^u$ such that $(C')^*=C'$.

Let $G'\subset G(D)$ is a component which is not self-dual. Then due to the transitive closure rule, the vertex set $C'=V(G')$ splits into two subsets $C'_1$ and $C'_2$ such that any two crossings from one subset are connected by an even edge, and two crossings from different subsets are connected by an odd edge. The subsets $C'_1$ and $C'_2$ are dual phratries of the finest phratry system $\mathfrak P^u$.
\end{remark}

Let us describe the finest phratry system for diagrams of a flat tangle.

Let $T=K_1\cup\cdots\cup K_n$ be an oriented flat tangle in the surface $F$ and $\mathfrak T$ be the category of flat diagrams of $T$. In order to distinguish phratries we need to refine the flat component and homotopy types.

\begin{definition}\label{def:refined_flat_types}
1. Let $D=D_1\cup\cdots\cup D_n$ be a diagram of the tangle $T$, and $v$ be a crossing of $D$. Then $v$ is an intersection point of components $D_i$ and $D_j$. We order the component according to the orientation as shown in Fig.~\ref{pic:refined_flat_component_type}. The \emph{refined flat component type} of the crossing $v$ is the ordered pair $\tau^f(v)=(i,j)$.

\begin{figure}[h]
\centering
  \includegraphics[width=0.12\textwidth]{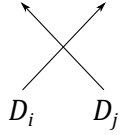}
  \caption{A crossing of refined flat component type $(i,j)$}\label{pic:refined_flat_component_type}
\end{figure}

Note that there are $n^2$ possible refined flat component types.

2. Let $i=j$ and $\partial D_i\ne\emptyset$. Let $s_i$ be the starting point of $D_i$. The crossing $v$ split the component $D_i$ into two halves $D^l_v$ and $D^r_v$ (Fig.~\ref{pic:knot_halves}) one of which is closed (it does not contain $s_i$) and the other is not (it contains $s_i$). We denote the closed half by $D^c_v$.

We define the \emph{refined order type} $o^f(v)$ of the crossing $v$ to be equal to $+1$ if $D^c_v=D^l_v$ and to be equal to $-1$ if $D^c_v=D^r_v$ (Fig.~\ref{pic:refined_flat_order_type})

\begin{figure}[h]
\centering
  \includegraphics[width=0.4\textwidth]{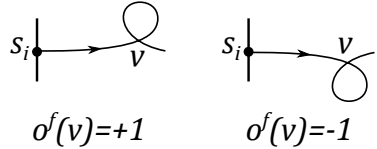}
  \caption{Refined order type}\label{pic:refined_flat_order_type}
\end{figure}

3. Let $i=j$ and $\partial D_i=\emptyset$.  Choose a non crossing point $z_i$ in the component $D_i$. Let $\kappa_{z_i}$ be the homotopy class of $D_i$ in the group $\pi_1(F,z_i)$.  Let $\bar\pi_D(F,z_i)$ be the quotient of $\pi_1(F,z_i)$ by the adjoint action of the element $\kappa_{z_i}$.

For the crossing $v$, choose a path $\gamma_{v,z_i}$ which connects $z_i$ with $v$ in $D_i$. Define the based left half as the loop $\hat D^l_{v,z_i}=\gamma_{v,z_i}D^l_v(\gamma_{v,z_i})^{-1}$ (Fig.~\ref{pic:based_left_half}). Analogously, $\hat D^r_{v,z_i}=\gamma_{v,z_i}D^r_v(\gamma_{v,z_i})^{-1}$.

\begin{figure}[h]
\centering
  \includegraphics[width=0.3\textwidth]{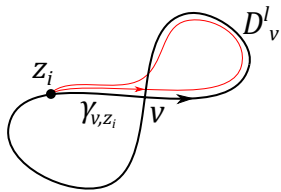}
  \caption{Based left half $\hat D^l_{v,z_i}$}\label{pic:based_left_half}
\end{figure}

The homotopy classes $\delta^l_{v,z_i}=[\hat D^l_{v,z_i}]$ and $\delta^r_{v,z_i}=[\hat D^r_{v,z_i}]$ in $\pi_1(F,z_i)$ depend on the choice of $\gamma_{v,z}$ but their images $\bar \delta^l_{v,z_i}$ and $\bar\delta^r_{v,z_i}$ in $\bar\pi_D(F,z_i)$ do not. Note that $\delta^l_{v,z_i}\delta^r_{v,z_i}=\kappa_{z_i}$ if $z_i\in D^r_v$ and $\delta^r_{v,z_i}\delta^l_{v,z_i}=\kappa_{z_i}$ if $z_i\in D^l_v$.

The element
\begin{equation}\label{eq:closed_refined_flat_homotopy_type}
  h^f_{D}(v)=\bar \delta^l_{v,z_i}\in \bar\pi_D(F,z_i)
\end{equation}
is called the \emph{refined flat homotopy type} of the self-crossing $v$ of the closed component $D_i$.

If $v$ is a self-crossing of a long component $D_i$ with the beginning point $s$ then the \emph{refined flat homotopy type} $h^f(v)$ of the crossing $v$ coincides with the homotopy type $h(v)$ defined by the formula~\eqref{eq:long_homotopy_type}:  $h^f_D(v)=h(v)\in\pi_1(F,s)$.

If $v$ is a mixed crossing of components $D_i$ and $D_j$, $i<j$ then the \emph{refined flat homotopy type} $h^f_D(v)$ of the crossing $v$ coincides with the homotopy type $h_{D,z_i,z_j}(v)$ defined by the formula~\eqref{eq:mixed_homotopy_type}: $h^f_{D}(v)=h_{D,z_i,z_j}(v)\in\bar\pi(F,z_i,z_j)$.
\end{definition}

We consider the following involutions on the sets which the types take values in:
\[
(i,j)^*=(j,i), \quad 1\le i,j\le n,
\]
for the refined flat component type,
\[
o^*=-o,\quad o\in\{-1,+1\},
\]
for the refined flat order type, and
\[
\bar x^*=\overline{\kappa_{z_i}x^{-1}},\quad \bar x\in\bar\pi_D(F,z_i)
\]
for the homotopy type of self-crossings of a closed component.

We define the involution to be the identity on the sets $\pi_1(F,s)$ and $\bar\pi(F,z_i,z_j)$ where the homotopy type of self-crossings of long components and mixed crossings takes value.

\begin{lemma}\label{lem:refined_flat_types_r2}
Let $v$ and $w$ be crossings of a flat tangle diagram $D$ to which a second Reidemeister move can be applied. Then
\begin{itemize}
  \item $\tau^f(w)=\tau^f(v)^*$;
  \item $o^f(w)=o^f(v)^*$ if $v$ and $w$ are self-crossings of a long component;
  \item $h_D(w)=h_D(v)^*$.
\end{itemize}
\end{lemma}

\begin{proof}
The first statement of the lemma follows from the definition of the refined flat component type (Fig.~\ref{pic:refined_flat_component_type_r2}).

\begin{figure}[h]
\centering
  \includegraphics[width=0.2\textwidth]{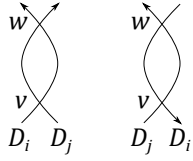}
  \caption{Mixed crossing in second Reidemeister moves}\label{pic:refined_flat_component_type_r2}
\end{figure}

If $v$ and $w$ a self-crossing of a long component then their refined order types are opposite by the definition of the type (see Fig.~\ref{pic:refined_flat_order_type_r2}).

\begin{figure}[h]
\centering
  \includegraphics[width=0.3\textwidth]{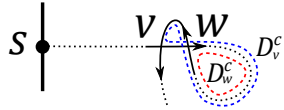}
  \caption{Refined order type in a second Reidemeister move}\label{pic:refined_flat_order_type_r2}
\end{figure}

Let $v$ and $w$ a self-crossing of a closed component. Then the left based half $\hat D^l_{w,z_i}$ is homotopic to the based right half $\hat D^r_{v,z_i}$ (Fig.~\ref{pic:refined_flat_homotopy_type_r2}).

\begin{figure}[h]
\centering\includegraphics[width=0.3\textwidth]{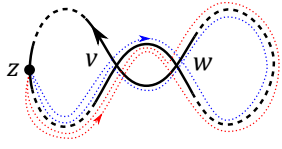}
\caption{Based halves $\hat D^r_{v,z}$ (blue) and  $\hat D^l_{w,z}$ (red)}\label{pic:refined_flat_homotopy_type_r2}
\end{figure}

Hence,
\[
h^f_{D,z_i}(w)=\overline{[\hat D^l_{w,z_i}]}=\overline{[\hat D^r_{v,z_i}]}=\overline{\kappa_{z_i}[\hat D^l_{w,z_i}]^{-1}}=\sigma\left(\overline{[\hat D^l_{v,z_i}]}\right)=\sigma\left(h^f_{D,z_i}(v)\right).
\]

By Theorem~\ref{thm:tribes_flat_tangle}, the homotopy types of $v$ and $w$ coincide if the crossings are mixed or they are self-crossings of a long component.
\end{proof}

Now, we can give a description of crossing phratries of flat tangle diagrams.

\begin{theorem}\label{thm:phratries_flat_tangles}
Let $T$ be an oriented flat tangle in the surface $F$ and $D$ be its diagram. Then two crossings $v$ and $w$ of the diagram $D$ belong to one phratry if and only if they have identical refined flat component types, refined order types (if $v$ and $w$ are self-crossing of a long component) and refined flat homotopy types.
\end{theorem}

\begin{proof}

Let $G(D)$ be the weighted graph from Remark~\ref{rem:phratry_graph}. We claim that for any crossings $v$ and $w$ which belong to one tribe (i.e. a connected component of $G(D)$) their refined flat component, order and homotopy types coincide if the edge weight $\epsilon(vw)$ is zero, and the types are dual (i.e. differ by the correspondent involutions) if the edge $vw$ is odd ($\epsilon(vw)=1$). This statement can be proved by the induction which defines the weights of the graph $G(D)$:
\begin{itemize}
\item $R_2$-duality is proved in Lemma~\ref{lem:refined_flat_types_r2};
\item Reidemeister translation follows from the fact that for any Reidemeister move $f\colon D\to D'$ and any vertices $v\in\V(D)$ and $v'\in\V(D)$ such that $v'=f_*(v)$ we have $\tau^f(v')=\tau^f(v)$, $o^f(v')=o^f(v)$ and $h^f_{D'}(v')=h^f_D(v)$;
\item transitive closure holds because the square of each of the involutions of the crossing types is identity.
\end{itemize}

Now, if $v$ and $w$ belong to one phratry the they are connected by an even edge of the graph $G(D)$. Thus, their refined flat types coincide.

Next, assume that the refined flat types of crossings $v$ and $w$ coincide. Then by definition $\bar\tau(v)=\bar\tau(w)$ and $\bar h(v)=\bar h(w)$. Hence, $v$ and $w$ belong to one tribe by Theorem~\ref{thm:tribes_flat_tangle}. This means $v$ and $w$ belong either to one phratry or to dual phratries.

If $v$ and $w$ are mixed crossings then the dual phratries are distinguished by the refined component type. Since $\tau^f(v)=\tau^f(w)$, they belong to one phratry.

If $v$ and $w$ are self-crossings of a long component then the dual phratries are distinguished by the refined order type. Since $o^f(v)=o^f(w)$, they belong to one phratry.

Let $v$ and $w$ are self-crossings of a closed component $D_i$. If $h_D(v)^*\ne h_D(v)$ then the dual phratries can be distinguished by the homotopy type. Since $h^f_D(v)=h^f_D(w)$, they belong to one phratry.

Assume that $h_D(v)^* = h_D(v)$. Then $\hat D^l_{v,z_i}=\hat D^r_{v,z_i}=\hat D^l_{w,z_i}=\hat D^r_{w,z_i}$. Lift the flat tangle diagram $D$ to a tangle diagram $\tilde D$ so that $v$ and $w$ be positive crossings in $\tilde D$. Then
\[
h_{\tilde D}(v)=\overline{[\hat {\tilde D}^+_{v,z_i}]}=\overline{[\hat D^{l/r}_{v,z_i}]}=\overline{[\hat D^{l/r}_{w,z_i}]}=\overline{[\hat {\tilde D}^{+}_{w,z_i}]}
=h_{\tilde D}(w).
\]
By Corollary~\ref{cor:phratries_surface_tangle} the crossings $v$ and $w$ belong to one phratry in $\tilde D$. Thus, they belong to one phratry in $D$.
\end{proof}

Let us describe the self-dual phratries of crossings. By Theorem~\ref{thm:phratries_flat_tangles} this can happen only for self-crossings of a closed component.

\begin{proposition}
Let $D$ be a flat tangle diagram, $D_i$ be a closed component of $D$ and $v$ be a self-crossing of $D_i$. Choose a non-crossing point $z_i$ in $D_i$. Let $\kappa_{z_i}\in \pi_1(F,z_i)$ be the homotopy class of $D_i$ and $\delta^l_{v,z_i}=[\hat D^l_{v,z_i}]\in\pi_1(F,z_i)$ be the homotopy class of a based left half at the crossing $v$. Then $v$ belongs to a self-dual phratry if  and only if
\begin{equation}\label{eq:self_dual_phratry}
  (\delta^l_{v,z_i})^2=\kappa_{z_i}.
\end{equation}
\end{proposition}

\begin{proof}
  For brevity denote $a=\delta^l_{v,z_i}$ and $b=\delta^r_{v,z_i}$. Without loss of generality assume that $z_i\in D^r_v$. Then $ab=\kappa_{z_i}$.

  By Theorem~\ref{thm:phratries_flat_tangles} the phratry of the crossing $v$ is self-dual if and only if $\bar a=\bar a^*$. Since $\bar a^*=\bar b$, the condition is $b=\kappa^l_{z_i}a\kappa^{-l}_{z_i}$ for some $l\in\Z$. If $l=0$ then $a=b$ and $\kappa_{z_i}=a^2$.

  Let $l>0$. Then $b(ab)^l=(ab)^la$. Consider the covering $\tilde F\to F$ which correspond to the subgroup $H=\left<a,b\right>\subset\pi_1(F,z_i)$ generated by $a$ and $b$. Then $\pi_1(\tilde F)=H$ and $H_1(\tilde F,\Z)=H/[H,H]$. In the group $H/[H,H]$ we have the relation $b+l(a+b)=l(a+b)+a$ which implies $a=b\in H/[H,H]$. Then the homology group $H_1(\tilde F,\Z)$ is cyclic, and the surface $\tilde F$ is a cylinder. Hence, $H=\pi_1(\tilde F)=H_1(\tilde F,\Z)$, and $a=b$ in $H$. Thus, $a=b$ in $\pi_1(F,z_i)$ and $\kappa_{z_i}=ab=a^2$.

The case $l<0$ is proved analogously. Thus, the self-duality condition is equivalent to  $\kappa_{z_i}=a^2$.
\end{proof}

\begin{corollary}
1. A closed component $D_i$ can have a self-dual phratry of self-crossings if and only if its homotopy type $\kappa_{z_i}$ is a square of an element of $\pi_1(F,z_i)$.

2. A closed component $D_i$ has at most one self-dual phratry.
\end{corollary}

The second statement of the corollary follows from the fact that any element in the fundamental group of a surface has at most one square root.

\section{Universal index on tangles in the surface $F$}\label{sect:universal_index_tribe}

Below we describe the universal index $\iota^u$ on tangles in the surface $F$.

Let $T=K_1\cup\cdots\cup K_n$ be a tangle in the surface $F$ and $\mathfrak T$ be its diagram category. Let $D=D_1\cup\cdots\cup D_n$ be a diagram of $T$.

By Theorem~\ref{thm:main_theorem} the component type $\tau$ is an index with coefficients in $\{1,\dots,n\}^2$. Then the universal index must distinguish the component type of crossings. Thus, we can consider crossings of different types separately.

\subsection{Self-crossings of a long component}

Let $D_i$ be a long component and $s_i\in\partial D_i$ be its beginning point.

The order type and the homotopy type define an index $o\times h$, $(o\times h)(v)=(o(v),h(v))$, with coefficients in $\Z_2\times\pi_1(F,s_i)$. Let us prove it is universal for the self-crossings of a long component.

\begin{proposition}\label{prop:long_index}
Let $\iota$ be an index with coefficients in a set $I$. Let $v$ and $w$ be self-crossings of the long component $D_i$. If $o(v)=o(w)$ and $h(v)=h(w)$ then $\iota(v)=\iota(w)$.
\end{proposition}
The proposition follows from Theorem~\ref{thm:main_theorem} and Proposition~\ref{prop:index_and_tribes}.

\begin{corollary}\label{cor:long_universal_index}
The index $o\times h$ is universal for self-crossings of the long component $D_i$.
\end{corollary}

\begin{proof}
Let $\iota$ be an index with coefficients in a set $I$. We need to construct a map $\psi\colon\Z_2\times\pi_1(F,s_i)\to I$ such that for any self-crossing $v$ of $D_i$ one has $\iota(v)=\psi\circ(o\times h)(v)$.

Let $\epsilon\in\{-1,+1\}$ and $\alpha\in\pi_1(F,s_i)$. By applying second Reidemeister moves, create a crossing $v$ such that $o(v)=\epsilon$ and $h(v)=\alpha$ as shown in Fig.~\ref{pic:homotopy_type_crossing} if the diagram $D$ has no such a crossing. Then set $\psi(\epsilon, \alpha)=\iota(v)$.

\begin{figure}[h]
\centering\includegraphics[width=0.5\textwidth]{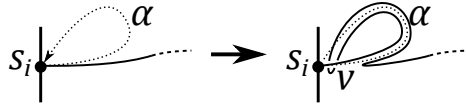}
\caption{A crossing with a given homotopy type $\alpha$}\label{pic:homotopy_type_crossing}
\end{figure}

By Proposition~\ref{prop:long_index} the map $\psi$ is well-defined, and by definition $\iota(v)=\psi\circ(o\times h)$.
\end{proof}

\subsection{Self-crossings of a closed component}

Let $D_i$ be a closed component. Unlike the long case, we can not use the group $\bar\pi_D(F,z_i)$ where the homotopy type takes values, as the coefficient set for a homotopy index because the point $z_i$ is not fixed.

Choose an arbitrary point $z^0_i\in F\setminus\partial F$. The following lemma shows that we can reduce (is some sense) the diagrams of the tangle $T$ to diagrams which contains the point $z^0_i$.

\begin{lemma}\label{lem:basing_morphism}
1. There exists a tangle diagram $\tilde D$ such that $z^0_i\in\tilde D_i$ and $\tilde D$ is obtained from $D$ by applying increasing second Reidemeister moves.

2. For any morphism $f\colon D\to D'$, i.e. a sequence of Reidemeister moves
\[
D=D^0\stackrel{f_1}{\rightarrow}D^1\rightarrow\cdots\rightarrow D^{n-1}\stackrel{f_n}{\rightarrow}D^n=D',
\]
 there exists a sequence of morphisms
 \[
 \tilde D^0\stackrel{\tilde f_1}{\rightarrow}\tilde D^1\rightarrow\cdots\rightarrow \tilde D^{n-1}\stackrel{\tilde f_n}{\rightarrow}\tilde D^n
 \]
such that $z^0_i\in\tilde D^k_i$ and $\tilde D^k$ is obtained from $D^k$ by applying increasing second Reidemeister moves, $k=0,\dots,n$.
\end{lemma}

\begin{proof}
1. The diagram $D$ is a projection of a tangle $T\subset F\times (0,1)$. Choose points $(z_i, s)\in T$ and $(z^0_i,s_0)\not\in T$ and connect them by a path $\gamma$ such that $\gamma\cap T=(z_i, s)$. Then pull a sprout from the point $(z_i, s)$ to the point $(z^0_i,s_0)$ along $\gamma$. We obtain a tangle $\tilde T$ (Fig.~\ref{pic:basing_diagram}). Its projection $\tilde D$ to $F$ obeys the required conditions.

\begin{figure}[h]
\centering\includegraphics[width=0.5\textwidth]{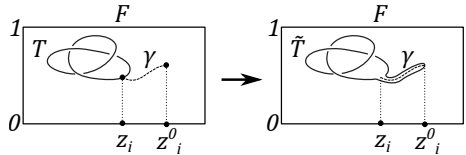}
\caption{The tangle $\tilde T$}\label{pic:basing_diagram}
\end{figure}

2. Denote the natural projection from $F\times(0,1)$ to $F$ by $p$. Let $T\subset F\times(0,1)$ be a tangle such that $D=p(T)$. Let $H_t\colon F\times[0,1]\to F\times[0,1]$, $t\in[0,1]$, be an isotopy which represents the morphism $f\colon D\to D'$. This means that $D^i=p(H_{t_i}(T))$ for some $t_0=0<t_1<\cdots<t_n=1$ such that for any $i=1,\dots,n$ the interval $(t_{i-1},t_i)$ contains exactly one singular value of the map $p\circ H_t\circ T$ which corresponds to the Reidemeister move $f_i$.

Push the tangle $T$ to the top by the map $(z, s)\mapsto (z, \frac 23+\frac s3)$ and replace the isotopy $H_t$ with the isotopy
\[
\hat H_t(z,s)=\left\{\begin{array}{cl}
                       H_t(z, 3s-2) & s\in[\frac 23,1], \\
                       H_{(3s-1)t}(z,0) & s\in[\frac 13, \frac 23], \\
                       (z,s) & s\in[0,\frac 13],
                     \end{array}\right.
\]
which is still in the ``abyssal'' layer $F\times[0,1/3]$. Then $\hat H_t$ also represents the morphism $f$.

Consider the tangle $\tilde T$ from the first part of the proof and ``drop the anchor'' from the point $(z^0_i,s_0)$ to the point $(z^0_i,\frac 13)$ and get a tangle $\hat T$ (Fig.~\ref{pic:basing_morphism}). Then define $\tilde D^i=p(\hat H_{t_i}(\hat T))$. By definitions of $\hat H_t$ and $\hat T$ the diagrams $D^i$ and $\tilde D^i$ differs by a sprout which ends in the point $p(z^0_i,\frac 13)=z^0_i$. One can pull this sprout with second Reidemeister moves. Thus, the diagrams $\tilde D^i$ satisfy the required conditions.
\begin{figure}[h]
\centering\includegraphics[width=0.5\textwidth]{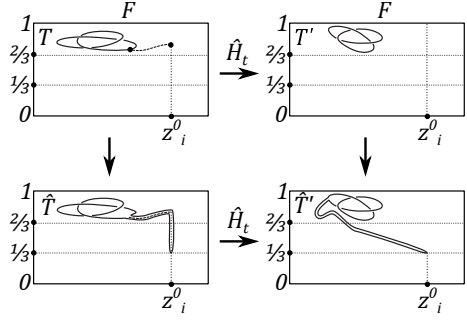}
\caption{The tangle $\hat T$}\label{pic:basing_morphism}
\end{figure}
\end{proof}

\begin{remark}\label{rem:sprout_homotopy_type}
The sprout which differs $\tilde T$ from $T$ goes along the curve $\gamma\gamma_0$ where $\gamma_0$ is the vertical segment connecting the points $(z^0_i,s_0)$ and $(z^0_i,\frac 13)$. Then the sprout which differs $\tilde T'$ from $T'$ goes along $\hat H_1(\gamma\gamma_0)$. Let $\gamma_1$ be defined by the formula $\gamma_1(t)=\hat H_{1-t}(z_i,s)$, $t\in[0,1]$. Then the path $\hat H_1(\gamma\gamma_0)$ is homotopic to the path $\gamma_1\gamma\gamma_0$, and the projection $p(\hat H_1(\gamma\gamma_0))$ is homotopic to $p(\gamma_1\gamma\gamma_0)=p(\gamma_1)p(\gamma)$.

Thus, the diagram $\tilde D$ is obtained from $D$ by pulling a sprout along a path in the homotopy class $p(\gamma)$, and $\tilde D'$ is obtained from $D'$ by pulling a sprout along  $p(\gamma_1)p(\gamma)$. In particular case, when $D$ contains $z^0_i$ one can choose $z_i=z^0_i$. Then $p(\gamma)$ is trivial, and the homotopy type of the sprout in $\tilde D'$ is $p(\gamma_1)$, i.e. the projection of the path that the point $(z_i,s)$ walks during the homotopy $\hat H_t$ (oriented in the opposite direction).
\end{remark}

Let us define the based diagram category $\mathfrak T_{z^0_i}$ which consists of the diagrams $D$ of the tangle $T$ for which the $i$-th component $D_i$ includes $z^0_i$ as a non-crossing point. The morphisms of $\mathfrak T_{z^0_i}$ are compositions of isotopies and Reidemeister moves between diagrams in $\mathfrak T_{z^0_i}$.

Fix a diagram $D^0=D^0_1\cup\cdots\cup D^0_n\in \mathfrak T_{z^0_i}$. Denote by $\mathfrak T^0_{z^0_i}$ the subcategory consisting of the diagrams in $\mathfrak T_{z^0_i}$ which are connected by a morphism with the diagram $D^0$.

Let us denote by $\bar\pi^0(F,z^0_i)$ the quotient of the group $\pi_1(F,z^0_i)$ by the adjoint action of the element $\kappa_{z^0_i}=[D^0_i]$. For any self-crossing $v$ of the closed component $D_i$ of a tangle diagram $D=D_1\cup\cdots\cup D_n\in\mathfrak T^0_{z^0_i}$, its homotopy type $h_{D,z^0_i}(v)$ is an element of $\bar\pi^0(F,z^0_i)$ (note that $\bar\pi^0(F,z^0_i)=\bar\pi_D(F,z^0_i)$ because $D_i$ and $D^0_i$ are homotopic as loops with the end $z^0_i$). Let us denote this element by $h^0_i(v)$.

\begin{proposition}\label{prop:based_closed_universal_index}
The map $h^0_i$ is the universal index for self-crossings of the closed component $D_i$ on the diagram category $\mathfrak T^0_{z^0_i}$.
\end{proposition}

\begin{proof}
  The map $h^0_i$ is an index with coefficients in $\bar\pi^0(F,z^0_i)$ by Theorem~\ref{thm:main_theorem}. Then we need to show that for any index $\iota$ with coefficients in a set $I$ there exists a unique map $\psi^0\colon \bar\pi^0(F,z^0_i)\to I$ such that $\iota(v)=\psi^0(h^0_i(v))$ for any self-crossing $v\in\V(D)$, $D\in \mathfrak T^0_{z^0_i}$, of the component $D_i$.

Like in corollary~\ref{cor:long_universal_index}, for an element $\alpha\in\bar\pi^0(F,z^0_i)$ we define the map $\psi^0$ by the formula $\psi^0(\alpha)=\iota(v)$ where $v$ is a crossing with the homotopy type $h^0_i(v)=h_{D,z^0_i}(v)=\alpha$. Let us prove that the map $\psi^0$ is well-defined.

Let $\alpha\in\bar\pi^0(F,z^0_i)$. Let $v\in\V(D)$ and $v'\in\V(D')$, $D,D'\in\mathfrak T^0_{z^0_i}$ be two crossings such that $h^0_i(v)=h^0_i(v')=\alpha$. Since $D,D'\in\mathfrak T^0_{z^0_i}$ there exist a morphism $f\colon D\to D'$ in the category $\mathfrak T^0_{z^0_i}$. Then there is a tangle $T\in F\times(0,1)$ and an isotopy $H_t\colon F\times[0,1]\to F\times[0,1]$, $t\in[0,1]$, which realizes $f$. That is $p(T)=D, p(H_1(T))=D'$ and for some $s_0\in(0,1)$  one has $(z^0_i,s_0)\in T_i$ where $T_i$ is the $i$-th component of $T$ and $H_t(z^0_i,s_0)=(z^0_i,s_0)$ for any $t$. Like in Lemma~\ref{lem:basing_morphism}, we can assume that the isotopy $H_t$ does not move the abyssal layer $F\times[0,\frac 13]$. Then we sink the point $(z^0_i,s_0)$ and tie a crossing with the homotopy type $\alpha$ in the abyssal layer (Fig.~\ref{pic:based_homotopy_index}). We get a tangle $\tilde T$ and a crossing $w$ in the diagram $\tilde D=p(\tilde T)$ such that $h^0_i(w)=\alpha$.

\begin{figure}[h]
\centering\includegraphics[width=0.5\textwidth]{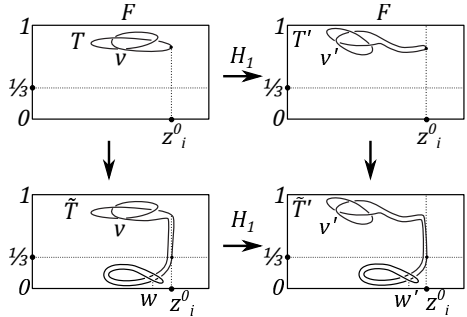}
\caption{The tangle $\tilde T$}\label{pic:based_homotopy_index}
\end{figure}

The isotopy $H_1$ does not disturb the crossing $w$. Then it corresponds to a crossing $w'$ in the diagram $\tilde D'=p(\tilde T')$ where $\tilde T'=H_1(\tilde T)$. By construction $h^0_i(w')=h^0_i(w)=\alpha$. By property (I0) we have $\iota(w)=\iota(w')$. On the other hand, by Theorem~\ref{thm:main_theorem} $h^0_i(v)=h^0_i(w)$ implies $v\sim w$ in $\tilde D$. Then $\iota(v)=\iota(w)$. Analogously, $\iota(v')=\iota(w')$. Hence, $\iota(v)=\iota(v')$.

Thus, we has proved that $\psi^0$ is well-defined. The equality $\iota(v)=\psi^0(h^0_i(v))$ follows from the definition.
\end{proof}

Now let us extend the homotopy index to all diagrams. Let $\iota$ be an index with coefficients in a set $I$ on the diagram category $\mathfrak T$.

Let $D\in \mathfrak T$ be a diagram of the tangle $T$. Choose a non-crossing point $z_i\in D_i$. There is a morphism in $\mathfrak T$ from $D^0$ to $D$. The morphism is realized by an isotopy $H_t$, $t\in[0,1]$ of $F\times[0,1]$. We can suppose that $H_1$ maps $z^0_i$ to the point $z_i$. More precisely, let $T^0\subset F\times(0,1)$ be a tangle such that $p(T^0)=D^0$ and $y^0_i=p^{-1}(z^0_i)\cap T^0$. Then $D=p(H_1(T^0))$ and we suppose that $p(H_1(y^0_i))=z_i$. Denote the point $H_1(y^0_i)$ by $y_i$.

By Lemma~\ref{lem:basing_morphism} there is a morphism $f\colon D^0\to\tilde D$ in $\mathfrak T^0_{z^0_i}$ where the diagram $\tilde D$ is obtained from $D$ by pulling a sprout from the point $z_i$ to the point $z^0_i$. By Remark~\ref{rem:sprout_homotopy_type} the sprout is pulled along a path which is homotopic to the path $\bar\gamma_D=p(H_{1-t}(y_i))$, $t\in[0,1]$.

Since the pulling process consists of increasing second Reidemeister moves, the set of crossings $\V(D)$ embeds in the set of crossings $\V(\tilde D)$. By property (I0) the value $\iota(v)$ of a crossing $v$ in $D$ coincides with the value $\iota(v)$ of this crossing in the diagram $\tilde D$. Since $\tilde D\in\mathfrak T^0_{z^0_i}$, by Proposition~\ref{prop:based_closed_universal_index} for any $v\in\tilde D$ $\iota(v)=\psi^0(h^0_i(v))$. By definition of $h^0_i$, for any $v\in\V(D)$
\[
h^0_i(v)=h_{\tilde D,z^0_i}(v)=[\hat D^+_{v,z^0_i}]=\bar\gamma_D^{-1}[\hat D^+_{v,z_i}]\bar\gamma_D=\bar\gamma_D^{-1}h_{D,z_i}(v)\bar\gamma_D.
\]

The elements $h^0_i(v)$ depends on $\gamma_D$ and hence on the isotopy $H_t$.

Let $H'_t$, $t\in[0,1]$, be another isotopy from $D^0$ to $D$ such that $H'_1(y^0_i)=y_i$. Denote the path $p(H'_{1-t}(y_i))$ by $\bar\gamma'_D$. Then $H^{-1}_t\circ H_t$ is an isotopy from the diagram $D^0$ to itself, and $\bar\gamma=\bar\gamma_D^{-1}\bar\gamma'_D$ is the homotopy class of the path of the base point $z^0_i$ under this isotopy.

On the other hand, if $H^0_t$ be an isotopy of $D^0$ to itself such that $H^0_1(y^0_i)=y^0_i$ and $\bar\gamma = p(H^0_t(y^0_i))$, $t\in[0,1]$, then the composition $H'_t=H_t\circ H^0_t$ is an isotopy from $D^0$ to $D$ with $\bar\gamma'_D=\bar\gamma_D\bar\gamma^{-1}$.

These reasonings lead us to the following definition.

\begin{definition}\label{def:inner_monodromy_group}
Let $D=D_1\cup\cdots\cup D_n$ be a tangle diagram, $D_i$ be a closed component of $D$ and $z\in D_i$ be a non-crossing point. Let $T=T_1\cup\cdots\cup T_n\subset F\times(0,1)$ be a tangle such that $p(T)=D$ and $p(T_i)=D_i$. Let $y=p^{-1}(z)\cap T_i$. Define the \emph{inner monodromy group} $IM_i(D,z)$ as the subgroup of $\pi_(F,z)$ consisting of the homotopy classes of the loop $\gamma(t)=p(H_t(y))$, $t\in[0,1]$, where $H_t\colon F\times[0,1]\to F\times[0,1]$ is an isotopy such that $H_1(T)=T$ and $H_1(y)=y$.
\end{definition}

\begin{definition}\label{def:closed_homotopy_index}
Let $\hat\pi_{D^0}(F,z^0_i)$ the quotient set of the group $\pi_1(F,z^0_i)$ by the adjoint action of the inner monodromy group $IM_i(D^0,z^0_i)$. For any crossing $v\in\V(D)$ of a diagram $D\in\mathfrak T$, we define its \emph{homotopy index} $h_i(v)$ as the element
\begin{equation}\label{eq:closed_homotopy_index}
h_i(v)=\widehat{\bar\gamma_D^{-1}[\hat D^+_{v,z_i}]\bar\gamma_D}=\widehat{\bar\gamma_D^{-1}h_{D,z_i}(v)\bar\gamma_D}\in\hat\pi_{D^0}(F,z^0_i).
\end{equation}
\end{definition}

\begin{proposition}\label{prop:universal_closed_index}
The map $h_i$ is the universal index for self-crossings of the closed component $D_i$ on the diagram category $\mathfrak T$.
\end{proposition}

\begin{proof}
The factorization by the adjoint action in $\hat\pi_{D^0}(F,z^0_i)$ ensures that $h_i(v)$ is well defined. The index property (I0) holds by definition and the property  (I2) follows from the correspondent property of the homotopy type $h_{D,z_i}$.

Let $\iota$ be an index with coefficients in a set $I$ on the diagram category $\mathfrak T$. We define a map $\psi\colon\hat\pi_{D^0}(F,z^0_i)\to I$ by the formula $\psi(\hat\alpha)=\iota(v)$, $\alpha\in\pi_1(F,z^0_i)$ where $v$ is a crossing of a diagram $D\in\mathfrak T^0_{z^0_i}$ such that $[\hat D^+_{v,z^0_i}]=\alpha$.

Let us prove that the map $\psi$ is well-defined. Let $v\in\V(D)$, $D\in\mathfrak T$, be a crossing. The reasonings above show that the value $\iota(v)$ is determined by the homotopy class $\bar\gamma_D^{-1}[\hat D^+_{v,z_i}]\bar\gamma_D$. Since we can choose any class $\bar\gamma_D$ induced from an isotopy from $D^0$ to $D$ and such classes differ by elements of the inner monodromy group, the value $\iota(v)$ can be restored from any representative of the orbit of the homotopy class by the adjoint action of $IM_i(D^0,z^0_i)$, that is from the element $h_i(v)$. Thus, $\psi$ if well-defined.

The equality $\iota=\psi\circ h_i$ follows from the definition.
\end{proof}

Let us enumerate some properties of the inner monodromy group.

\begin{proposition}\label{prop:inner_monodromy_group}
Let $D$ be a tangle diagram of the tangle $T$, $D_i$ be a closed component of $D$ and $z\in D_i$ be a non-crossing point.
\begin{enumerate}
  \item Let $D'\in\mathfrak T$ be another diagram of $T$ and $z'\in D'_i$ be a non-crossing point in the $i$-th component. Then $IM_i(D',z')\simeq IM_i(D,z)$.
  \item Let $\kappa_z=[D_i]\in\pi_1(F,z)$ be the homotopy class of the component $D_i$. Then $\kappa_z\in IM_i(D,z)$.
  \item $IM_i(D,z)\subset C(\kappa_z)$ where $C(\kappa_z)=\{\alpha\in\pi_1(F,z) \mid \alpha\kappa_z\alpha^{-1}=\kappa_z\}$ is the centralizer of the element $\kappa_z$.
\end{enumerate}
\end{proposition}

\begin{proof}
1. Since $D$ and $D'$ are diagrams of one tangle, there is a morphism $f\colon D\to D'$. Let $H_t$, $t\in[0,1]$, be an isotopy of $F\times[0,1]$ which realizes the morphism $f$. That means $D=p(T)$ for a tangle $T$ and $D'=p(H_1(T))$. We can also assume that $z'=p(H_1(y))$ where $y=p^{-1}(z)\cap T$. Let $\gamma(t)=p(H_t(y))$, $t\in[0,1]$ be the path which the point $z$ pass during the isotopy. Then the isomorphism
\[
\gamma_*\colon\pi_1(F,z)\to\pi_1(F,z'),\quad \alpha\mapsto\gamma^{-1}\alpha\gamma,
\]
maps the subgroup $IM_i(D,z)$ onto the subgroup $IM_i(D',z')$.

2. Let $T\subset F\times(0,1)$ be a tangle such that $p(T)=D$. Consider an isotopy $H_t$ which pushes the $i$-th component $T_i$ along itself so that the points of $T_i$ make a full turn, and does not move the other components. Then the corresponding loop $\gamma\in IM_i(D,z)$ coincides with $\kappa_z$.

3.  Let $T\subset F\times(0,1)$ be a tangle such that $p(T)=D$ and let $y\in T$ be the  point such that $p(y)=z$. Consider an isotopy $H_t$, $t\in[0,1]$, such that $H_1(T)=T$ and $H_1(y)=y$. Let $\gamma(t)=p(H_t(y))$ and $\gamma^{\le t}$ be the part of $\gamma$ connecting $z$ with $z(t)=p(H_t(y))$. Let $\kappa(t)\in\pi(F,z(t))$ be the homotopy type of $p(H_t(T_i))$. Then $\tilde\kappa(t)=\gamma^{\le t}\kappa(t)(\gamma^{\le t})^{-1}$, $t\in[0,1]$ is a homotopy between $\kappa_z$ and $\gamma\kappa_z\gamma^{-1}$. Hence, $\gamma\kappa_z\gamma^{-1}=\kappa_z$ and $\gamma\in C(\kappa_z)$. Thus, $IM_i(D,z)\subset C(\kappa_z)$.
\end{proof}

\begin{remark}
If $\kappa_z$ is not contractible then $C(\kappa_z)=\left<\alpha\right>$ is a cyclic group generated by a prime homotopy class $\alpha$ such that $\kappa_z=\alpha^k$ for some $k\in\Z$.

If $\kappa_z$ is contractible then $C(\kappa_z)=\pi_1(F,z)$. But the inner monodromy group can differ from the centralizer as the following example shows (Fig.~\ref{pic:pretzel_knot}). Consider a knot $K$ in a surface of genus $2$. The homotopy type of $K$ is trivial. Let us show that $IM(K,z)=\left<\alpha\right>$ where $\alpha$ is the homotopy of the simple curve close to $K$ which separates the surface into two equal parts.
The curve $\alpha$ belongs to $IM(K,z)$ because one can drag the knot along $\alpha$.
\begin{figure}[h]
\centering\includegraphics[width=0.5\textwidth]{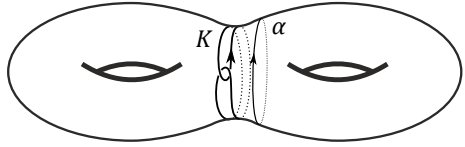}
\caption{A knot in a surface of genus $2$}\label{pic:pretzel_knot}
\end{figure}

Consider the linking invariant~\cite{Nif} (index polynomial) $lk_K(h^0)\in\Z[\pi_1(F,z)]$ corresponding to the homotopy index $h^0$. The knot has two negative crossings with the homotopy index equal to $\alpha$. Then $lk_K(h^0)=-\alpha-\alpha^{-1}$.

Assume $\gamma\in IM(K,z)$. Then then there is a diagram $K'$ obtained from $K$ by pulling a sprout along $\gamma$ which is connected to $K$ by an isotopy which does not move the base point $z$. Then $lk_{K'}(h^0)=lk_K(h^0)$.

On the other hand, the crossings have the homotopy index $\gamma^{-1}\alpha^\pm\gamma$ in $K'$, so $lk_{K'}(h^0)=-\gamma^{-1}\alpha\gamma-\gamma^{-1}\alpha^{-1}\gamma$. Then $\gamma^{-1}\alpha\gamma=\alpha^\pm$. The equality $\gamma^{-1}\alpha\gamma=\alpha^{-1}$ is impossible because $\alpha\ne-\alpha$ in $H_1(F,\Z)$. Hence, $\gamma^{-1}\alpha\gamma=\alpha$, i.e. $\gamma$ and $\alpha$ commute. Then $\gamma=\alpha^k$ for some $k\in\Z$.
\end{remark}

\subsection{Mixed crossings}

Let us define a universal index for mixed crossing. Let $i,j\in\{1,\dots,n\}$ and $i\ne j$.

Fix a diagram $D^0=D^0_1\cup\cdots\cup D^0_n$ of the tangle $T$. Choose con-crossing points $z^0_i$ and $z^0_j$ in the components $D^0_i$ and $D^0_j$. If the component $D^0_i$ ($D^0_j$) is long then we set $z^0_i=s^0_i$  ($z^0_j=s^0_j$) where  $s^0_i$ ($s^0_j$) is the beginning point of the component $D^0_i$ ($D^0_j$). 

Let us define the inner monodromy group for the case of mixed crossing.

\begin{definition}\label{def:inner_monodromy_gropup}
Let $D=D_1\cup\cdots\cup D_n$ be a tangle diagram, $D_i$ and $D_j$ be a closed component of $D$ and $z_i\in D_i$ and $z_j\in D_j$ be non-crossing points. If a component is long then we choose the beginning point of the component for $z_i$ or $z_j$.
Let $T=T_1\cup\cdots\cup T_n\subset F\times(0,1)$ be a tangle such that $p(T)=D$, $p(T_i)=D_i$ and $p(T_j)=D_j$. Let $y_i=p^{-1}(z_i)\cap T_i$ and $y_j=p^{-1}(z_j)\cap T_j$. Define the \emph{inner monodromy group} $IM_{ij}(D,z_i,z_j)$ as the subgroup of $\pi_1(F,z_i)\times\pi_1(F,z_j)$ consisting of the pairs $(\gamma_i,\gamma_j)$ where $\gamma_i(t)=p(H_t(y_i))$ and $\gamma_j(t)=p(H_t(y_j))$, $t\in[0,1]$. Here $H_t\colon F\times[0,1]\to F\times[0,1]$ is an isotopy such that $H_1(T)=T$, $H_1(y_i)=y_i$ and $H_1(y_j)=y_j$.
\end{definition}

\begin{proposition}\label{prop:mixed_inner_monodromy_group}
Let $D=D_1\cup\cdots\cup D_n$ be a tangle diagram, $D_i$ and $D_j$ be components of $D$ and $z_i\in D_i$ and $z_j\in D_j$ be non-crossing points chosen as above. Let $\kappa_{z_i}\in\pi_1(F,z_i)$ be the homotopy class of $D_i$ if the component is closed, and $1$ otherwise. The element $\kappa_{z_j}\in\pi_1(F,z_j)$ is defined analogously.

\begin{enumerate}
  \item Let $D'\in\mathfrak T$ be another diagram of $T$ and $z'_i\in D'_i$, $z'_j\in D'_j$ be a non-crossing point in the $i$-th component. Then $IM_{ij}(D',z'_i,z'_j)\simeq IM_{ij}(D,z_i,z_j)$.
  \item $(\kappa_{z_i},1), (1,\kappa_{z_j}) \in IM_{ij}(D,z_i,z_j)$.
  \item $p_1(IM_{ij}(D,z_i,z_j))=IM_i(D,z_i)$ and $p_2(IM_{ij}(D,z_i,z_j))=IM_j(D,z_j)$ where $p_1, p_2$ are the natural projections from $\pi_1(F,z_i)\times\pi_1(F,z_j)$ to the groups $\pi_1(F,z_i)$ and $\pi_1(F,z_j)$
\end{enumerate}
\end{proposition}
The proposition is proved analogously to Proposition~\ref{prop:inner_monodromy_group}.

\begin{remark}\label{rem:general_inner_monodromy_group}
Let $D=D_1\cup\cdots\cup D_n$ be a diagram of a tangle $T\subset F\times(0,1)$ and $z_1,\dots,z_n$ be non-crossing points in the components $D_1,\dots, D_n$. As usual, for long components suppose that the chosen points coincide with the beginning points. Then one can define the \emph{general inner monodromy group} $IM(D,z_1,\dots,z_n)$ as the subgroup in $\prod_{i=1}^n \pi_1(F,z_i)$ which consists of tuples $(\gamma_i)_{i=1}^n$ of the homotopy classes of the paths which the chosen points $z_i$ (more accurately, their liftings $y_i$ to $T$) pass during some isotopy $H_t$, $t\in[0,1]$, such that $H_1(T)=T$ and $H_1(y_i)=y_i$, $i=1,\dots,n$. The general inner monodromy group has the following properties:

\begin{enumerate}
  \item For another diagram $D'$ of $T$ and another choice of points $z'_1,\dots, z'_n$ the groups $IM(D',z'_1,\dots,z'_n)$ and $IM(D,z_1,\dots,z_n)$ are isomorphic.
  \item $(1,\dots,1,\kappa_{z_i},1,\dots,1)\in IM(D,z_1,\dots,z_n)$ where $\kappa_{z_i}=[D_i]\in\pi_1(F,z_i)$.
  \item $IM(D,z_1,\dots,z_n)\subset \prod_{i=1}^n C(\kappa_{z_i})$ (for long components we set $C(\kappa_{z_i})=\{1\}$).
  \item $p_i(IM(D,z_1,\dots,z_n))=IM_i(D,z_i)$ and $p_{ij}(IM(D,z_1,\dots,z_n))=IM_{ij}(D,z_i,z_j)$ where $p_i, p_{ij}$ are the natural projections from $\prod_{i=1}^n \pi_1(F,z_i)$ to $\pi_1(F,z_i)$ and $\pi_1(F,z_i)\times\pi_1(F,z_j)$.
\end{enumerate}
\end{remark}

Now we can define the homotopy index $h$ for the mixed crossings of the component type $(i,j)$.

\begin{definition}\label{def:mixed_homotopy_index}
Let $\pi_1(F,z^0_i,z^0_j)$ be the set of homotopy classes of continuous paths in $F$ from $z^0_i$ to $z^0_j$. Let $\hat\pi_{D^0}(F,z^0_i,z^0_j)$ be the quotient of $\pi_1(F,z^0_i,z^0_j)$ by the action $(\gamma_1,\gamma_2)\cdot \beta=(\gamma_1\beta\gamma_2^{-1})$ of the subgroup $IM_{ij}(D,z_i,z_j)$.

Let $D=D_1\cup\cdots\cup D_n$ be a diagram of the tangle $T$ and $z_i\in D_i$ and $z_j\in D_j$ be non-crossing points which coincide with the beginning point of a component if it is long. Let $H_t$ be an isotopy from $D$ to $D^0$ which moves $z_i$ to $z^0_i$ and $z_j$ to $z^0_j$, and $\bar\gamma_{D,i}$ and $\bar\gamma_{D,j}$ be the corresponding paths.

For any crossing $v$ in $D$ of the component type $(i,j)$, we define its \emph{homotopy index} $h_{ij}(v)$ as the element
\begin{equation}\label{eq:mixed_homotopy_index}
h_{ij}(v)=\widehat{\bar\gamma_{D,i}^{-1}h_{D,z_i,z_j}(v)\bar\gamma_{D,j}}\in\hat\pi_{D^0}(F,z^0_i,z^0_j)
\end{equation}
where $h_{D,z_i,z_j}(v)$ is defined by the formula~\eqref{eq:mixed_homotopy_type}.
\end{definition}

\begin{proposition}\label{prop:universal_mixed_index}
The map $h_{ij}$ is the universal index for mixed crossings of the component type $(i,j)$ on the diagram category $\mathfrak T$.
\end{proposition}

The proposition is proved analogously to Proposition~\ref{prop:universal_closed_index}.

\subsection{Universal index}\label{subsect:universal_index}

Collecting the results of Corollary~\ref{cor:phratries_surface_tangle} and Propositions~\ref{prop:universal_closed_index} and~\ref{prop:universal_mixed_index}, we come to the following theorem.

\begin{theorem}\label{thm:universal_index_tangles}
Let $T=T_1\cup\cdots\cup T_n$ be a tangle in the surface $F$ and $D^0=D^0_1\cup\cdots\cup D^0_n$ be its diagram.
Let $I^u=\bigsqcup_{i,j=1}^n I^u_{ij}$ where
\[
I^u_{ij}=\left\{\begin{array}{cl}
\Z_2\times\pi_1(F,s_i),& i=j\mbox{ and } \partial D_i\ne\emptyset,\\
\hat\pi_{D^0}(F,z^0_i),& i=j\mbox{ and } \partial D_i=\emptyset,\\
\hat\pi_{D^0}(F,z^0_i,z^0_j),& i\ne j.
\end{array}
\right.
\]
For a diagram $D$ of the tangle $T$, define a map $\iota^u\colon\V(D)\to I^u$ by the formula
\[
\iota^u(v)=\left\{\begin{array}{cl}
(o(v),h(v)),& \tau(v)=(i,i)\mbox{ and } \partial D_i\ne\emptyset,\\
h_i(v),& \tau(v)=(i,i)\mbox{ and } \partial D_i=\emptyset,\\
h_{ij}(v),& \tau(v)=(i,j),\ i\ne j.
\end{array}
\right.
\]
Then $\iota^u$ is the universal index on the diagram category of the tangle $T$.
\end{theorem}

\begin{corollary}\label{cor:classical_trivial_indices}
Let $K$ be a classical knot. Then there are no nontrivial indices on crossings of diagrams of $K$.
\end{corollary}

\begin{proof}
For a classical knots we have $F=\R^2$ or $S^2$. In both cases $\pi_1(F,z)=1$, hence, the homotopy index $h$ is trivial. For a knot, the component type $\tau$ is trivial too, and there is no order type. Thus, the universal index is constant, and all other indices are constant too.
\end{proof}

For classical links and long knots we have analogous statements.

\begin{corollary}\label{cor:long_classical_knots}
The only nontrivial index for long classical knots is the order type $o$.
\end{corollary}

\begin{corollary}\label{cor:classical_links}
The component type $\tau$  is the universal index of classical links.
\end{corollary}

\begin{example}\label{exa:annulus}
Let $F=S^1\times [0,1]$ be an annulus. Its fundamental group is $\pi_1(F)=H_1(F)=\Z$. Since the group is commutative, the adjoint action on it is trivial.

Let $K$ be a knot in $F$ and $D$ be a diagram of $K$. For a crossing $v\in\V(D)$, its homotopy index coincides with the homology class of the signed half $D^+_v$: $h(v)=[D^+_v]\in H_1(F)=\Z$.
\end{example}

\begin{corollary}
The map $h$ with coefficients in $\Z$ is the universal index for knots in the annulus.
\end{corollary}
Note that $sgn\cdot h$ is an ordered parity in the sense of~\cite{Nwp}.

Let $K$ be a long knot in the annulus $F$. Then $\partial K\subset \partial F=S^1\times\{0,1\}$. Let $D$ be a diagram of $K$ and $v$ be a crossing of $D$. The vertex $v$ splits the diagram $D$ into two halves, one of which is closed. We denote the close half by $D^c_v$.

The homotopy index of the crossing $v$ is the homology class of its closed half: $h(v)=[D^c_v]\in H_1(F)=\Z$.

\begin{corollary}
The map $h\times o\colon \V(D)\to \Z\times\Z_2$ where $o$ is the order type, is the universal index for long knots in the annulus.
\end{corollary}

\begin{example}\label{exa:torus}
Let $F=T^2$ be the torus. Then $\pi_1(F)=H_1(F)=\Z^2$. Since the group is commutative then the adjoint action on it is trivial.

Let $K$ be a knot in the (thickened) torus. Then the homotopy index of crossings diagrams is the homology type of their positive half: $h(v)=[D^+_v]\in H_1(T^2)=\Z^2$. Thus, we have a universal $\Z^2$-valued index for knots in the torus.

Note that $sgn\cdot h$ is an oriented parity in the sense of~\cite{Nwp}.
\end{example}

\subsection{Universal index of flat tangles in the surface $F$}

Let us formulate a result analogous to that of Section~\ref{subsect:universal_index}. The flat case is simpler because closed components of a flat tangle are essentially free loops in the surface $F$, and long components are curves in $F$ defined up to homotopy with fixed ends.

Let $T=K_1\cup\cdots\cup K_n$ be a flat tangle in the surface $F$ and $\mathfrak T$ be the category of flat diagrams of $T$.

Fix a diagram $D^0=D^0_1\cup\cdots\cup D^0_n\in\mathfrak T$ and choose points $z^0_i\in D^0_i$, $i=1,\dots,n$. If some component $D^0_i$ is long then we set $z^0_i=s^0_i$ where $s^0_i$ is the beginning point of the component $D^0_i$. Let $\kappa^0_i=[D^0_i]$ be the homotopy class of the $i$-th component in the group $\pi_1(F,z^0_i)$ if $D^0_i$ is closed, and $\kappa^0_i=1\in\pi_1(F,z^0_i)$ is $D^0_i$ is long, $i=1,\dots,n$. Let $C^0_i$ be the subgroup in $\pi_1(F,z^0_i)$ which is equal to the centralizer $C(\kappa^0_i)$ is the component $D^0_i$ is closed, and equal to $\{1\}$ if the component is long.

Let us define sets $\bar I^u_{ij}$, $i,j\in\{1,\dots,n\}$ as follows
\[
I^u_{ij}=\left\{\begin{array}{cl}
\hat{\bar\pi}_{D^0}(F,z^0_i),& i=j,\\
\hat{\bar\pi}_{D^0}(F,z^0_i,z^0_j),& i\ne j.
\end{array}
\right.
\]
where $\hat{\bar\pi}_{D^0}(F,z^0_i)=\pi_1(F,z^0_i)$ for a long component $D^0_i$;
$\hat{\bar\pi}_{D^0}(F,z^0_i)$ is the quotient set of the group $\pi_1(F,z^0_i)$ by the adjoint action of the subgroup $C^0_i$ and the involution $\sigma(x)=\kappa^0_i x^{-1}$, $x\in\pi_1(F,z^0_i)$, for a closed component $D^0_i$; and $\hat{\bar\pi}_{D^0}(F,z^0_i,z^0_j)$ is the quotient set  $C^0_i\backslash\pi_1(F,z_i,z_j)/C^0_j$. Let $\bar I^u=\bigsqcup_{i,j=1}^n \bar I^u_{ij}$.

Given a flat diagram $D=D_1\cup\cdots\cup D_n$ of the flat tangle $T$, choose non-crossing points $z_i\in D_i$, $i=1,\dots,n$  which coincide with the beginning points for the long components of $D$. For each closed component $D_i$, choose a path $\bar\gamma_{D,i}$ such that it connects $z_i$ with $z^0_i$ and $\bar\gamma_{D,i}^{-1}[D_i]\bar\gamma_{D,i}=\kappa^0_i$. Such paths exist because $D$ is homotopic to $D^0$ in $F$.

We define a map $\bar\iota^u\colon \V(D)\to \bar I^u$ as follows: for a crossing $v$ of the diagram $D$ the value $\bar\iota^u(v)$ is the image of the flat homotopy type $\bar h(v)$ (see Definition~\ref{def:flat_types}) under the natural projection of $\pi_1(F,s^0_i)$ or $\bar{\bar\pi}_{D^0}(F,z^0_i)$ to $\hat{\bar\pi}_{D^0}(F,z^0_i)$, or the natural projection of ${\bar\pi}_{D^0}(F,z^0_i,z^0_j)$ to $\hat{\bar\pi}_{D^0}(F,z^0_i,z^0_j)$.

\begin{theorem}\label{thm:universal_flat_index}
The map $\bar\iota$ with the coefficients in $\bar I^u$ is the universal index on the diagrams of the flat tangle $T$.
\end{theorem}

The proof follows the proof of Corollary~\ref{cor:long_universal_index} and Propositions~\ref{prop:universal_closed_index} and~\ref{prop:universal_mixed_index}. First we prove that the flat homotopy type $\bar h$ is the universal index on the diagram category which consists of diagrams $D$ that can be connected with $D^0$ by a homotopy which does not move the base points $z^0_i$. Then we pass to the general category $\mathfrak T$.

Note that the flat analogues of the inner monodromy groups $IM_i(D^0,z^0_i)$ and $IM_{ij}(D^0,z^0_i,z^0_j)$ coincide with the centralizers $C(\kappa^0_i)$
and $C(\kappa^0_i)\times C(\kappa^0_j)$ correspondingly.

\section{Index for virtual and flat knots}\label{sect:virtual_flat_knots}

\subsection{Virtual knots}

Recall that \emph{oriented virtual links} $\mathcal L$ are equivalence classes of pairs $(F,L)$ where $F$ is an oriented closed surface and $L\subset F\times(0,1)$ is an oriented link in thickening of $F$, considered modulo isotopies, diffeomorphisms and stabilization/destabilization operations~\cite{CKS}, see Fig.~\ref{pic:stabilization}.

\begin{figure}[h]
\centering\includegraphics[width=0.5\textwidth]{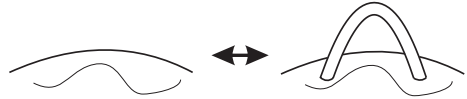}
\caption{Stabilization operation}\label{pic:stabilization}
\end{figure}

Given a virtual link $\mathcal L$, we consider its diagram category $\mathfrak L$ whose objects are pairs $(F,D)$ where $F$ is a surface and $D\subset F$ is a link diagram (a general projection of $L$ for some representative $(F,L)$ of $\mathcal L$). The morphisms in $\mathfrak L$ are compositions of surface isotopies, Reidemeister moves, diffeomorphisms of pairs $(F,D)$ and (de)stabilizations.

G. Kuperberg proved the following theorem~\cite{Kuperberg}.

\begin{theorem}\label{thm:Kuperberg}
Let $(F,L)$ and $(F',L')$ be two representatives of a virtual link $\mathcal L$ and the genus of $F$ and $F'$ be minimal among all representatives of $\mathcal L$. Then there exists an orientation preserving diffeomorphism $\Phi\colon F\times[0,1]\to F'\times[0,1]$ such that $\Phi(L)=L'$.
\end{theorem}

Kuperberg's theorem allows us to introduce the diagram subcategory $\mathfrak L_{min}$ consisting of diagrams $D$ of $\mathcal L$ in surfaces $F$ of minimal genus. The morphisms in $\mathfrak L_{min}$ are isotopies, Reidemeister moves and diffeomorphisms.

Note that we don't have stabilization morphisms in $\mathfrak L_{min}$, and the surfaces $F$ of all objects $(F,D)$ in $\mathfrak L_{min}$ are diffeomorphic.
Unlike diagram categories of links in a fixed surface, the category $\mathfrak L_{min}$ has new morphisms --- diffeomorphisms of thickened surfaces.

Let us describe the universal index in the category $\mathfrak L_{min}$.

\begin{definition}
Let $D=D_1\cup\cdots\cup D_n$ be a link diagram in a surface $F$ of minimal genus. Let $L=L_1\cup\cdots\cup L_n$ be the corresponding link in $F\times(0,1)$. Choose non-crossing points $z_i\in D_i$, $i=1,\dots,n$. Let $y_i=p^{-1}(z_i)\cap L_i$ be the corresponding points in $L$. Let $\widetilde{Sym}(L,y_1,\dots,y_n)$ be the subgroup in the diffeomorphism group $Diff^+(F\times[0,1])$ which consists of diffeomorphisms $\Phi\colon F\times[0,1]\to F\times[0,1]$ such that $\Phi(L_i)=L_i$ and $\Phi(y_i)=y_i$, $i=1,\dots,n$.

Every diffeomorphism $\Phi\in\widetilde{Sym}(L,y_1,\dots,y_n)$ induces the automorphisms $\Phi_*\in Aut(\pi_1(F,z_i))$ and $\Phi_*\in Aut(\pi_1(F,z_i,z_j))$. The group
\[
Sym(D,z_i)=\{\Phi_* \mid \Phi\in\widetilde{Sym}(L,y_1,\dots,y_n)\}\subset Aut(\pi_1(F,z_i))
\]
is the \emph{symmetry group} for the component $D_i$ and the group

\[
Sym(D,z_i,z_j)=\{\Phi_*\mid \Phi\in\widetilde{Sym}(L,y_1,\dots,y_n)\}\subset Aut(\pi_1(F,z_i,z_j))
\]
is the \emph{symmetry group} for the components $D_i$ and $D_j$.
\end{definition}

\begin{proposition}\label{prop:symmetry_group_property}
Let $D=D_1\cup\cdots\cup D_n\in\mathfrak L_{min}$ be a link diagram and $z_i\in D_i$, $i=1,\dots,n$, be non-crossing points. Then
\begin{enumerate}
 \item The groups $Sym(D,z_i)$ and $Sym(D,z_i,z_j)$ are well defined, and for any other diagram $D'$ and a set of non-crossing points $z'_i\in D'_i$, $i=1,\dots,n$, we have $Sym(D,z_i)\simeq Sym(D',z'_i)$ and  $Sym(D,z_i,z_j)\simeq Sym(D',z'_i,z'_j)$ for all $1\le i,j\le n$.
 \item Let $\gamma\in IM_i(D,z_i)$. Then $Ad_\gamma\in Sym(D,z_i)$ where $Ad_\gamma(\alpha)=\gamma\alpha\gamma^{-1}$, $\alpha\in\pi_1(F,z_i)$, is the adjoint action.
 \item Let $(\gamma_i,\gamma_j)\in IM_{ij}(D,z_i,z_j)$. Then $Ad_{\gamma_i,\gamma_j}\in Sym(D,z_i,z_j)$ where $Ad_{\gamma_i,\gamma_j}(\alpha)=\gamma_i\alpha\gamma_j^{-1}$, $\alpha\in\pi_1(F,z_i,z_j)$.
 \item Let $\kappa_{z_i}=[D_i]\in\pi_1(F,z_i)$. Then for any $\phi\in Sym(D,z_i)$ we have $\phi(\kappa_{z_i})=\kappa_{z_i}$.
\end{enumerate}
\end{proposition}

\begin{proof}
1. Let $D'=D'_1\cup\cdots\cup D'_n$ be a diagram in a surface $F'$, and $L'=L'_1\cup\cdots\cup L'_n\in F'\times(0,1)$ be a corresponding link, i.e. $p(L')=D'$. Let $y'_i=p^{-1}(z'_i)\cap L'$, $i=1,\dots,n$.  By Kuperberg's theorem there exists a diffeomorphism $\Psi\colon F\times[0,1]\to F'\times[0,1]$ such that $\Psi(L_i)=L'_i$. We can suppose that $\Psi(y_i)=y'_i$ for all $i$. Then for any $\Phi\in\widetilde{Sym}(L,y_1,\dots,y_n)$ the composition $\Phi'=\Psi\Phi\Psi^{-1}$ belongs to $\widetilde{Sym}(L',y'_1,\dots,y'_n)$ and vice versa. Then we have a sequence of isomorphisms of groups
\[
\pi_1(F,z_i)\stackrel{p_*}{\leftarrow}\pi_1(F\times[0,1],y_i)\stackrel{\Psi_*}{\rightarrow}\pi_1(F'\times[0,1],y'_i)\stackrel{p_*}{\rightarrow}\pi_1(F',z'_i)
\]
which induces an isomorphism $\Psi_*\colon Aut(\pi_1(F,z_i))\to Aut(\pi_1(F',z'_i))$. The isomorphism $\Psi_*$ identifies $Sym(D,z_i)$ and  $Sym(D',z'_i)$.

Analogously, the groups $Sym(D,z_i,z_j)$ and  $Sym(D',z'_i,z'_j)$ are also isomorphic.

2. Let $\gamma\in IM_i(D,z_i)$. By Definition~\ref{def:inner_monodromy_group} there exist an isotopy $H_t$, $t\in[0,1]$, such that $H_1(L)=L$ and $H_1(y_i)=y_i$. Without loss of generality, we can assume that $H_1(y_j)=y_j$ for all $j$. Then $H_1\in \widetilde{Sym}(L,y_1,\dots,y_n)$ and $(H_1)_*=Ad_{\gamma^{-1}}\in Sym(D,z_i)$. Hence, $Ad^{-1}_{\gamma^{-1}}=Ad_\gamma\in Sym(D,z_i)$.

3. The third statement is proved analogously to the previous one.

4. Let $\Phi\in\widetilde{Sym}(L,y_1,\dots,y_n)$. Then
\[
\Phi_*(\kappa_{z_i})=\Phi_*p_*[L_i]=p_*\Phi_*[L_i]=p_*[\Phi(L_i)]=p_*[L_i]=\kappa_{z_i}.
\]
\end{proof}

Let $\mathcal L$ be a virtual link. Fix a diagram $D^0=D^0_1\cup\cdots\cup D^0_n$ in a surface $F^0$ of minimal genus. Choose non-crossing points $z^0_i\in D^0_i$, $i=1,\dots,n$. Let $L^0=L^0_1,\cup\cdots\cup L^0_n\subset F^0\times(0,1)$ be a link such that $p(L^0)=D^0$. Let $y^0_i=p^{-1}(z^0_i)\cap L^0_i$, $i=1,\dots,n$.

Denote the sets of orbits by action of the symmetry groups by
\[
\pi^v_{D^0}(F,z^0_i)=\pi_1(F^0,z^0_i)/Sym(D^0,z^0_i),\quad i=1,\dots,n,
\]
and
\[
\pi^v_{D^0}(F,z^0_i,z^0_j)=\pi_1(F^0,z^0_i,z^0_j)/Sym(D^0,z^0_i,z^0_j),\quad i,j=1,\dots,n.
\]

\begin{definition}\label{def:virtual_homotopy_index}
Let $(D,F)\in\mathfrak L_{min}$ be a virtual link diagram. Let $L\subset F\times(0,1)$ be a link such that $p(L)=D$. Choose non-crossing points $z_i\in D_i$ and let $y_i\in L_i$ be such that $p(y_i)=z_i$, $i=1,\dots,n$.

Let $\Psi_D\colon F^0\times[0,1]\to F\times[0,1]$ be a diffeomorphism such that $\Psi_D(L^0_i)=L_i$ and $\Psi_D(y^0_i)=y_i$.

Let $v$ be a self-crossing of a component $D_i$. The \emph{homotopy index} $h^v_i(v)$ of the crossing $v$ is defined by the formula
\begin{equation}\label{eq:virtual_closed_homotopy_index}
h^v_i(v)=(\Psi_D)^{-1}_*h_{D,z_i}(v)=(\Psi_D)^{-1}_*[\hat D^+_{v,z_i}]\in\pi^v_{D^0}(F^0,z^0_i).
\end{equation}

Let $v\in\V(D)$ be a mixed crossing of component type $\tau(v)=(i,j)$. The \emph{homotopy index} $h^v_{ij}(v)$ of the crossing $v$ is defined by the formula
\begin{equation}\label{eq:virtual_mixed_homotopy_index}
h^v_{ij}(v)=(\Psi_D)^{-1}_*h_{D,z_i,z_j}(v)\in\pi^v_{D^0}(F^0,z^0_i,z^0_j).
\end{equation}
where the homotopy class $h_{D,z_i,z_j}(v)\in\pi_1(F,z_i,z_j)$ is defined by the formula~\eqref{eq:mixed_homotopy_type}.
\end{definition}

Note that by definition of the groups $\pi^v_{D^0}(F^0,z^0_i)$ and $\pi^v_{D^0}(F^0,z^0_i,z^0_j)$, the homotopy index is well defined.

\begin{theorem}\label{thm:universal_virtual_index}
Let $\mathcal L$ be a virtual link and $(F^0,D^0)\in\mathfrak L_{min}$. Let
\[
I^v=\bigsqcup_{i=1}^n \pi^v_{D^0}(F^0,z^0_i) \sqcup \bigsqcup_{i\ne j} \pi^v_{D^0}(F^0,z^0_i,z^0_j).
\]
For a diagram $(F,D)\in\mathfrak L_{min}$ define a map $\iota^v\colon\V(D)\to I^v$ by the formula
\[
\iota^v(v)=\left\{\begin{array}{cl}
h^v_i(v), & \tau(v)=(i,i),\\
h^v_{ij}(v), & \tau(v)=(i,j),\ i\ne j.
\end{array}\right.
\]
Then $\iota^v$ is the universal index on the diagram category $\mathfrak L_{min}$.
\end{theorem}

\begin{proof}
The index properties for $\iota^v$ for isotopies and Reidemeister moves follow from the corresponding properties of the homotopy type. The property (I0) for diffeomorphisms follows from the definition of the homotopy index $h^v$.

Let us prove that $\iota^v$ is universal. Let $\iota$ be an index with coefficients in a set $I$ on the diagram category $\mathfrak L_{min}$. We need to construct a map $\phi\colon I^v\to I$ such that $\iota(v)=\psi(\iota^v(v))$ for any crossing $v\in\V(D)$, $(F,D)\in\mathfrak L_{min}$.

Firstly, consider the case of self-crossings of an $i$-th component. Let $\alpha\in\pi_1(F^0,z^0_i)$. Pull a sprout from an arc near $z^0_i$ along a representative path for $\alpha$ and get a diagram ${D^0}'$ with a self-crossing $v\in\V(D')$ of the $i$-th component such that $[\widehat{({D^0}')}^+_{v,z^0_i}]=\alpha$ (Fig.~\ref{pic:test_crossing}). Note that the diagram ${D^0}'$ is obtained from $D^0$ by increasing second Reidemeister moves. Set $\psi(\alpha)=\iota(v)$.

\begin{figure}[h]
\centering\includegraphics[width=0.5\textwidth]{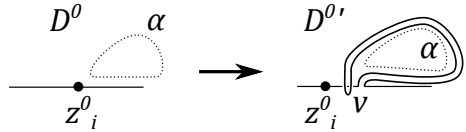}
\caption{Construction of the crossing $v$}\label{pic:test_crossing}
\end{figure}

Let us show that $\psi$ is well defined. If $w$ another crossing such that $[\hat D^+_{w,z^0_i}]=\alpha$ then $h_i(v)=h_i(w)\in\hat\pi_{D^0}(F^0,z^0_i)$, hence, $\iota(v)=\iota(w)$ by Proposition~\ref{prop:universal_closed_index}.

Let ${L^0}'$ be a link which corresponds to the diagram ${D^0}'$. We can suppose that the link ${L^0}'$ is obtained from $L^0$ by pulling a sprout from $y^0_i$ to a close neighbourhood of $y^0_i$.

Let $\Phi\in\widetilde{Sym}(L^0,y^0_1,\dots,y^0_n)$. Since $\Phi(y^0_i)=y^0_i$, the link $\Phi({L^0}')$ differ from $\Phi(L^0)=L^0$ by a sprout, and the diagram ${D^0}''=\Phi({D^0}')$ differs by a sprout from $D^0$. The diffeomorphism $\Phi$ maps the crossing $v$ to a crossing $w=\Phi_*(v)\in\V({D^0}'')$. The homotopy class of the based half of the crossing $w$ is equal to $[\widehat{({D^0}'')}^+_{w,z^0_i}]=\Phi_*([\widehat{({D^0}')}^+_{v,z^0_i}])=\Phi_*(\alpha)$.
On the other hand, by the property (I0), $\iota(w)=\iota(v)$. Thus, $\psi(\Phi_*(\alpha))=\psi(\alpha)$, and $\psi$ is well defined.

Let $(F,D)\in\mathfrak L_{min}$ and $v\in\V(D)$. Let $L\subset F\times(0,1)$ be a link such that $p(L)=D$ and $\Psi_D$ be a diffeomorphism which maps $L^0$ to $L$. Let $\alpha=\Psi^{-1}_*([\hat D^+_{v,z_i}])\in\pi_1(F^0, z^0_i)$. By definition the orbit of $\alpha$ in $\pi^v_{D^0}(F^0,z^0_i)$ is the homotopy index $h^v_i(v)$.

Create a crossing $w\in\V({D^0}')$ such that $[\widehat{({D^0}')}^+_{w,z^0_i}]=\alpha$ as before. Let $D'=p(\Psi({L^0}'))$. The diagram $D'$ differs from $D$ by a sprout. Let $w'=\Psi_*(w)\in\V(D')$. By construction $h_{D',z_i}(w')=\Psi_*(\alpha)=h_{D,z_i}(v)=h_{D',z_i}(v)$. Then $w'$ and $v$ belong to one tribe in $D'$ and by Proposition~\ref{prop:index_and_tribes} and the property (I0) $\iota(v)=\iota(w')=\iota(w)=\psi(\iota^v(v))$.
\end{proof}

\begin{example}\label{exa:virtual_link_torus}
Consider the link $D=D_1\cup D_2$ in the torus in Fig.~\ref{pic:motion_group_example}. The link is not classical, so the surface genus is minimal.

\begin{figure}[h]
\centering\includegraphics[width=0.2\textwidth]{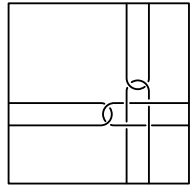}
\caption{A link in the torus}\label{pic:motion_group_example}
\end{figure}

The homotopy group of the torus does not depend on the base point and is equal to $\pi_1(T^2,z)=H_1(T^2,\Z)=\Z^2$.
The symmetry groups are $Sym_1(D)=Sym_2(D)=Sym_{12}(D)=Sym_{21}(D)=\Z_2$ (the nontrivial element of the groups acts like $x\mapsto -x$ on $\pi_1(T^2)=\Z^2$). Thus, $\pi^v_D=\Z_2^2$. That means there are 16 types of the crossings in diagrams in the torus of the virtual link. The crossings in Fig.~\ref{pic:motion_group_example} present 4 different types (the crossings with the same component type have the same index value).
\end{example}

\begin{remark}\label{rem:virtual_link_invariants}
Although we have described the index only for the subcategory $\mathfrak L_{min}$ in the category $\mathfrak L$ of all diagrams of the virtual link, this is enough to calculate link invariants which use an index in their construction. Indeed, if $Inv$ is an invariant (e.g. an index polynomial) which relies on an index $\iota$ with coefficients in a set $I$ on the diagrams of a link $\mathcal L$, then $Inv$ can be calculated on a diagram $(F,D)$ of the minimal genus. For such a diagram, $\iota=\psi\circ\iota^v$ for some $\psi\colon I^v\to I$, hence, the invariant $Inv$ is expressed by the index $\iota^v$, i.e. by the component index $\tau$ and homotopy index $h^v$.
\end{remark}

\subsection{Flat knots}

A \emph{flat link diagram} is a pair $(F,D)$ where $F$ is an oriented closed connected surface and $D\subset F$ is an embedded $4$-valent graph. A \emph{flat link} is an equivalence class of flat link diagrams modulo isotopies of the surface, Reidemeister moves, orientation preserving diffeomorphisms of pairs $(F,D)$ and (de)stabilizations. We assume that the components of a flat link are numbered and oriented.

An analogue of Kuperberg's theorem is the following statement~\cite{Fr} (see also~\cite{IMN_mon} for the knot case).

\begin{theorem}\label{thm:flat_Kuperberg}
  Let $(F,D)$, $D=D_1\cup\cdots\cup D_n$,  and $(F',D')$, $D'=D'_1\cup\cdots\cup D'_n$, be representatives of a flat link $\mathcal L$ with $n$ components such that the genus of $F$ and $F'$ is minimal among all diagrams of $\mathcal L$. Then there is an orientation preserving diffeomorphism $\Phi\colon F\to F'$ such that for all $i=1,\dots,n$  the free loops $\Phi(D_i)$ and $D'_i$ are homotopic in $F'$.
\end{theorem}

Let $\mathcal L$ be a flat link. Consider the category $\mathfrak L_{min}$ whose objects are the diagrams $(F,D)$ of $\mathcal L$ of minimal genus, and the morphisms are compositions of isotopies, Reidemeister moves and diffeomorphisms. By Theorem~\ref{thm:flat_Kuperberg} any two diagrams $(D,F)$ and $(D',F')$ in $\mathfrak L_{min}$ are connected with a morphism of $\mathfrak L_{min}$.

We will define the universal index on the diagram category $\mathfrak L_{min}$.

\begin{definition}\label{def:flat_symmetry_group}
Let $(F,D)$, $D=D_1\cup\cdots\cup D_n$, be a flat link diagram. Consider the following subgroup in $Diff^+(F)$:
\[
\widetilde{Sym}^f(D)=\{\Phi\in Diff^+(F)\mid \Phi(D_i)\sim D_i,\ i=1,\dots,n \}.
\]

Choose non-crossing points $z_i\in D_i$, $i=1,\dots,n$. Let $\kappa_{z_i}=[D_i]\in\pi_1(F,z_i)$.

Let $\Phi\in\widetilde{Sym}^f(D)$. By definition, there exist paths $\gamma_i$ from $z_i$ to $\Phi(z_j)$ such that $\gamma_i\Phi(\kappa_{z_i})\gamma_i^{-1}=\kappa_{z_i}$. Note that the path $\gamma_i$ is defined up to multiplication by elements from $IM_i(D,z_i)$.  Denote the map
\[
\alpha\mapsto \gamma_i\Phi(\alpha)\gamma_i^{-1},\quad \alpha\in\pi_1(F,z_i),
\]
by $\Phi_{\gamma_i}\in Aut(\pi_1(F,z_i))$. The set
\[
Sym^f(D,z_i)=\{ \Phi_{\gamma_i} \mid \Phi\in\widetilde{Sym}^f(D),\ \gamma_i\Phi(\kappa_{z_i})\gamma_i^{-1}=\kappa_{z_i}\}\subset Aut(\pi_1(F,z_i))
\]
is called the \emph{flat symmetry group} of the component $D_i$.

Denote the map
\[
\alpha\mapsto \gamma_i\Phi(\alpha)\gamma_j^{-1},\quad \alpha\in\pi_1(F,z_i,z_j),
\]
by $\Phi_{\gamma_i,\gamma_j}\in Aut(\pi_1(F,z_i,z_j))$. The set
\[
Sym^f(D,z_i,z_j)=\{ \Phi_{\gamma_i,\gamma_j} \mid \Phi\in\widetilde{Sym}^f(D)\}\subset Aut(\pi_1(F,z_i))
\]
is called the \emph{flat symmetry group} of the components $D_i$, $D_j$.
\end{definition}

\begin{proposition}\label{prop:flat_symmetry_group_property}
Let $D=D_1\cup\cdots\cup D_n\in\mathfrak L_{min}$ be a flat link diagram and $z_i\in D_i$, $i=1,\dots,n$, be non-crossing points. Then
\begin{enumerate}
 \item The sets $Sym^f(D,z_i)$ and $Sym^f(D,z_i,z_j)$ are groups.
 \item For any other diagram $D'$ and a set of non-crossing points $z'_i\in D'_i$, $i=1,\dots,n$, we have $Sym^f(D,z_i)\simeq Sym^f(D',z'_i)$ and  $Sym^f(D,z_i,z_j)\simeq Sym^f(D',z'_i,z'_j)$ for all $1\le i,j\le n$.
 \item Let $\gamma\in IM_i(D,z_i)$. Then $Ad_\gamma\in Sym^f(D,z_i)$ where $Ad_\gamma(\alpha)=\gamma\alpha\gamma^{-1}$, $\alpha\in\pi_1(F,z_i)$, is the adjoint action.
 \item Let $(\gamma_i,\gamma_j)\in IM_{ij}(D,z_i,z_j)$. Then $Ad_{\gamma_i,\gamma_j}\in Sym^f(D,z_i,z_j)$ where $Ad_{\gamma_i,\gamma_j}(\alpha)=\gamma_i\alpha\gamma_j^{-1}$, $\alpha\in\pi_1(F,z_i,z_j)$.
 \item For any $\phi\in Sym^f(D,z_i)$ we have $\phi(\kappa_{z_i})=\kappa_{z_i}$.
\end{enumerate}
\end{proposition}

\begin{proof}
1. Let $\Phi_{\gamma_i}, \Phi'_{\gamma'_i}\in Sym^f(D,z_i)$. Then
\[
\Phi_{\gamma_i}\circ \Phi'_{\gamma'_i}=(\Phi\circ\Phi')_{\gamma_i\Phi(\gamma_i')}\in Sym^f(D,z_i)
\]
and $(\Phi_{\gamma_i})^{-1}=(\Phi^{-1})_{\Phi^{-1}(\gamma_i^{-1})}$.

2. Assume first that $D'\subset F'$ and there exists a diffeomorphism $\Psi\colon F\to F'$ such that $D'_i=\Psi(D_i)$ and $\Psi(z_i)=z'_i$, $i=1,\dots,n$. Then the map
\[
\Phi\mapsto \Psi\circ\Phi\circ\Psi^{-1}
\] identifies the groups $\widetilde{Sym}^f(D)$ and $\widetilde{Sym}^f(D')$, and the map
\[
\Phi_{\gamma_i}\mapsto (\Psi\circ\Phi\circ\Psi^{-1})_{\Psi(\gamma_i)}
\]
identifies the groups $Sym^f(D,z_i)$ and $Sym^f(D',z'_i)$.

Now, let $D'\subset F$ and $D'_i\sim D_i$, $i=1,\dots,n$. Choose paths $\beta_i$ from $z_i$ to $z'_i$ such that $\beta_i\kappa'_{z'_i}\beta_i^{-1}=\kappa_{z_i}$. Let $H_t\colon F\to F$, $t\in[0,1]$ be an isotopy which moves the points $z_i$ to $z'_i$ along the paths $\beta_i$. Then $H_1(\kappa_{z_i})=\kappa'_{z'_i}\in\pi_1(F,z'_i)$. The isomorphism between $Sym^f(D,z_i)$ and $Sym^f(D',z'_i)$ is given by the formula $\Phi_{\gamma_i}\mapsto (H_1\circ\Phi\circ H_1^{-1})_{H_1(\gamma_i)}$.

The isomorphisms $Sym^f(D,z_i,z_j)\simeq Sym^f(D',z'_i,z'_j)$ are proved analogously.

3. Let $\gamma\in IM_i(D,z_i)$. Consider an isotopy $H_t$, $t\in[0,1]$, which pulls the point $z_i$ along the path $\gamma^{-1}$. Then $Ad_{\gamma}=(H_1)_1$ where $1\in\pi_1(F,z_i)$ is the trivial loop.

4. The fourth statement is proved analogously to the previous one.

5. The last statement holds by definition.
\end{proof}

Let $\mathcal L$ be a flat link. Fix a diagram $(F^0,D^0)\in\mathfrak L_{min}$, $D^0=D^0_1\cup\cdots\cup D^0_n$. Choose non-crossing points $z^0_i\in D^0_i$, $i=1,\dots,n$. Let $\kappa^0_{z^0_i}=[D^0_i]\in\pi_1(F^0,z^0_i)$.

Denote the sets of orbits by action of the symmetry groups by
\[
\pi^f_{D^0}(F,z^0_i)=\pi_1(F^0,z^0_i)/Sym^f(D^0,z^0_i)/\sigma,\quad i=1,\dots,n,
\]
and
\[
\pi^f_{D^0}(F,z^0_i,z^0_j)=\pi_1(F^0,z^0_i,z^0_j)/Sym^f(D^0,z^0_i,z^0_j),\quad i,j=1,\dots,n.
\]
The involution $\sigma$ on $\pi_1(F^0,z^0_i)$ is defined by the formula $\sigma(\alpha)=\kappa^0_{z^0_i}\alpha^{-1}$.

\begin{definition}\label{def:flat_homotopy_index}
Let $(D,F)\in\mathfrak L_{min}$ be a flat link diagram. Choose non-crossing points $z_i\in D_i$, $i=1,\dots,n$. Let $\kappa_{z_i}=[D_i]\in\pi_1(F,z_i)$.

By Theorem~\ref{thm:flat_Kuperberg} there exist a diffeomorphism $\Psi_D\colon F^0\to F$ and paths $\gamma_i$ from $z_i$ to $\Psi_D(z^0_i)$   such that $\gamma_i^{-1}\kappa_{z_i}\gamma_i=(\Psi_D)_*(\kappa^0_{z^0_i})\in\pi_1(F,\Psi_D(z^0_i))$.

Let $v$ be a self-crossing of a component $D_i$. The \emph{flat homotopy index} $h^f_i(v)$ of the crossing $v$ is defined by the formula
\begin{equation}\label{eq:flat_closed_homotopy_index}
h^f_i(v)=(\Psi_D)_{\gamma_i}^{-1}h_{D,z_i}(v)=(\Psi_D)^{-1}_{\gamma_i}[\hat D^+_{v,z_i}]\in\pi^f_{D^0}(F^0,z^0_i).
\end{equation}
where $(\Psi_D)_{\gamma_i}\colon \pi_1(F^0,z^0_i)\to \pi_1(F,z_i)$ is given by the formula
\[\alpha\mapsto \gamma_i\cdot (\Psi_D)_*(\alpha)\cdot\gamma_i^{-1}.
\]

Let $v\in\V(D)$ be a mixed crossing of component type $\tau(v)=(i,j)$. The \emph{homotopy index} $h^v_{ij}(v)$ of the crossing $v$ is defined by the formula
\begin{equation}\label{eq:flat_mixed_homotopy_index}
h^f_{ij}(v)=(\Psi_D)^{-1}_{\gamma_i,\gamma_j} h_{D,z_i,z_j}(v)\in\pi^v_{D^0}(F^0,z^0_i,z^0_j).
\end{equation}
where the homotopy class $h_{D,z_i,z_j}(v)\in\pi_1(F,z_i,z_j)$ is defined by the formula~\eqref{eq:mixed_homotopy_type} and the map $(\Psi_D)_{\gamma_i,\gamma_j}$ is defined by the formula
\[
\alpha\mapsto \gamma_i\cdot (\Psi_D)_*(\alpha)\cdot\gamma_j^{-1}.
\]
\end{definition}

\begin{theorem}\label{thm:universal_flat_link_index}
Let $\mathcal L$ be a flat link and $(F^0,D^0)\in\mathfrak L_{min}$. Let
\[
I^f=\bigsqcup_{i=1}^n \pi^f_{D^0}(F^0,z^0_i) \sqcup \bigsqcup_{i\ne j} \pi^f_{D^0}(F^0,z^0_i,z^0_j).
\]
For a diagram $(F,D)\in\mathfrak L_{min}$ define a map $\iota^f\colon\V(D)\to I^f$ by the formula
\[
\iota^f(v)=\left\{\begin{array}{cl}
h^f_i(v), & \tau(v)=(i,i),\\
h^f_{ij}(v), & \tau(v)=(i,j),\ i\ne j.
\end{array}\right.
\]
Then $\iota^f$ is the universal index on the diagram category $\mathfrak L_{min}$.
\end{theorem}

The theorem is proved analogously to Theorem~\ref{thm:universal_virtual_index}.

\begin{example}
Consider the flat knot $K$ if a surface $F$ of genus $2$ presented in Fig.~\ref{pic:flat_knot_example}. The diagram has a $\Z_5$-symmetry, and any diffeomorphism $\Phi\in\widetilde{Sym}^f(K)$ is isotopic to a rotation of the diagram. The internal monodromy group is $IM(K)=C(\kappa)=\langle\kappa\rangle\subset \pi_1(F,z)$ where $\kappa=[K]$.
 \begin{figure}[h]
\centering\includegraphics[width=0.3\textwidth]{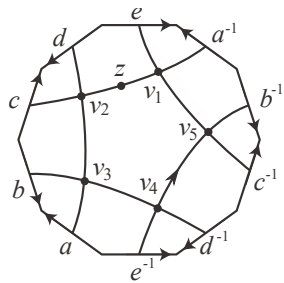} \caption{A flat knot in a surface of genus two}\label{pic:flat_knot_example}
 \end{figure}

The fundamental group has a presentation with generators $a,b,c,d,e$ (a generator is a loop which intersects the corresponding edge once and does not intersect other edges) and relations $abcde=1$ and $a^{-1}b^{-1}c^{-1}d^{-1}e^{-1}=1$. The class $\kappa$ is equal to $cebda$. The homotopy indices of the crossings are $h^f(v_1)=c$, $h^f(v_2)=a$, $h^f(v_3)=a^{-1}da$, $h^f(v_4)=a^{-1}d^{-1}bda$, $h^f(v_5)=cec^{-1}$.

Let $\Phi\in \widetilde{Sym}^f(K)$  be the counter clockwise rotation by the angle $\frac{2\pi}5$. Let $\gamma$ be the path in $K$ from $z$ to $\Phi(z)$ which has the opposite orientation to that of $K$. Then $\Phi_\gamma(a)=a^{-1}da$,  $\Phi_\gamma(b)=a^{-1}ea$,  $\Phi_\gamma(c)=a$,  $\Phi_\gamma(d)=a^{-1}ba$,  $\Phi_\gamma(e)=a^{-1}ca$.  Then $\Phi_\gamma(h^f(v_1))=h^f(v_2)$, $\Phi_\gamma(h^f(v_2))=h^f(v_3)$, $\Phi_\gamma(h^f(v_3))=h^f(v_4)$, $Ad_{\kappa}\Phi_\gamma(h^f(v_4))=h^f(v_5)$ and $\Phi_\gamma(h^f(v_5))=h^f(v_1)$. We see that the homotopy indices of all crossings coincide.
\end{example}


\begin{thebibliography}{99}
%

\bibitem {CKS} J.\,S.~Carter, S.~Kamada, and M.~Saito, Stable equivalence of knots on surfaces and virtual knot cobordisms {\em J. Knot Theory
Ramifications} {\bf 11}:3 (2002) 311--322.

\bibitem{Cheng} Z. Cheng, A polynomial invariant of virtual knots, \emph{Proc. Amer. Math. Soc.} {\bf 142}:2 (2014) 713--725.

\bibitem{Cheng3} Z. Cheng, A transcendental function invariant of virtual knots, {\em J. Math. Soc. Japan} {\bf 69}:4 (2017) 1583--1599.

\bibitem{Cheng2} Z. Cheng, The chord index, its definitions, applications and generalizations, \emph{Canad. J. Math.} {\bf 73}:3 (2021) 597--621.

\bibitem{CGX} Z. Cheng, H. Gao, M. Xu, Some remarks on the chord index,  \emph{J. Knot Theory Ramifications} {\bf 29}:10 (2020) 2042003.

\bibitem{CFGMX} Z. Cheng, D.A. Fedoseev, H. Gao, V.O. Manturov, M. Xu, From chord parity to chord index, \emph{ J. Knot Theory Ramifications} {\bf 29}:13 (2020) 2043004.


\bibitem{Dye} H. Dye, Smoothed invariants, \emph{J. Knot Theory Ramifications} {\bf 21}:13 (2012) 1240003.

\bibitem{FK} L.C. Folwaczny, L.H. Kauffman, A linking number definition of the affine index polynomial and applications, \emph{ J. Knot Theory Ramifications} {\bf 22}:12 (2013) 1341004.

\bibitem{Fr} D. Freund, Intersections of virtual multistrings, Thesis (Ph.D.)–Dartmouth College. 2018.

\bibitem{H} A. Henrich, A sequence of degree one vassiliev invariants for virtual knots, {\em J. Knot Theory Ramifications} {\bf 19}:4 (2010) 461--487.

\bibitem{ILL} Y.H. Im, K. Lee, S.Y. Lee, Index polynomial invariant of virtual links, \emph{J. Knot Theory Ramifications} {\bf 19} (2010) 709--725.


%
%
\bibitem{IMN_mon}  D.\,P.~Ilyutko, V.\,O.~Manturov, I.\,M.~Nikonov, Parity in knot theory and graph links, {\em J. Math. Sci.} {\bf 193}:6 (2013) 809--965.

\bibitem{Jeong} M.-J. Jeong, A zero polynomial of virtual knots, {\em J. Knot Theory Ramifications} {\bf 25}:1 (2016) 1550078.

\bibitem{KK}
N.~Kamada and S.~Kamada, ``Abstract link diagrams and virtual knots'', {\it Knot Theory
and Its Ramifications J.} {\bf 9}:1, 93--109 (2000).

\bibitem{K2} L. H. Kauffman, An affine index polynomial invariant of virtual knots, \emph{J. Knot Theory
Ramifications} {\bf 22}:4 (2013) 1340007.

\bibitem{KPV} K. Kaur, M. Prabhakar, A. Vesnin, Two-variable polynomial invariants of virtual knots
arising from flat virtual knot invariants, \emph{J. Knot Theory Ramifications} {\bf 27}:13 (2018) 1842015.

\bibitem{Kim} J. Kim, The affine index polynomial invariant of flat virtual knots, \emph{J. Knot Theory Ramifications} {\bf 23}:14 (2014) 1450073.

\bibitem{Kuperberg}
G.~Kuperberg, ``What is a Virtual Link?'', {\it Algebraic and Geometric Topology} {\bf
3}, 587--591 (2003).

%
%
%
%
%
%
%

\bibitem{Nwp}{I.\,M. Nikonov, Weak parities and functorial maps, \emph{J. Math. Sci.} {\bf 214}:5 (2016) 699--717.}

\bibitem{Npf}{I.\,M. Nikonov, Parity functors, arxiv:2109.12230.}

\bibitem{Nbm}{I.\,M. Nikonov, Parity on based matrices, arxiv:2110.04915.}

\bibitem{Nif}{I.\,M. Nikonov, An intersection formula for parities on virtual knots, arxiv:2110.08392.}

\bibitem{T} V.~Turaev, Virtual strings, {\em Ann. Inst. Fourier} {\bf 54}:7  (2004) 2455--2525.

\bibitem{T2} V.~Turaev, Cobordism of knots on surfaces, {\em J. Topol} {\bf 1}:2 (2008) 285--305.

\bibitem{Xu} M.Xu, Writhe polynomial for virtual links, arXiv:1812.05234.

\end{thebibliography}
\end{document}